\documentclass[10pt]{article}
\usepackage{latexsym, amssymb, amsmath, amsthm, amscd}
\usepackage[all,cmtip]{xy}
\usepackage{amsmath}
\usepackage{amsthm}
\usepackage{amssymb}
\usepackage{amsfonts}
\usepackage{amsrefs}
\usepackage{color,authblk}
\usepackage{tikz}

\usepackage{amscd} 
\usepackage{comment}
\usepackage{mathrsfs}
\usepackage[all]{xy}
\usepackage{graphicx}
\usepackage{authblk}
\usepackage{hyperref}
\usepackage{fullpage}
\usepackage[all,cmtip]{xy}
\usepackage{relsize}
\usepackage{amsmath,amscd}
\usepackage{amssymb}
\usepackage{stackrel}
\usepackage{tikz-cd}
\usepackage{tikz}
\usetikzlibrary{arrows}
\usetikzlibrary{matrix}
\usepackage[mathscr]{euscript}

\usepackage[all]{xy}

\usepackage{hyperref}

\numberwithin{equation}{section} \DeclareMathSizes{2}{10}{12}{13}
\makeatletter
\newcommand*{\doublerightarrow}[2]{\mathrel{
  \settowidth{\@tempdima}{$\scriptstyle#1$}
  \settowidth{\@tempdimb}{$\scriptstyle#2$}
  \ifdim\@tempdimb>\@tempdima \@tempdima=\@tempdimb\fi
  \mathop{\vcenter{
    \offinterlineskip\ialign{\hbox to\dimexpr\@tempdima+1em{##}\cr
    \rightarrowfill\cr\noalign{\kern.5ex}
    \rightarrowfill\cr}}}\limits^{\!#1}_{\!#2}}}

\newcommand{\leftrarrows}{\mathrel{\raise.75ex\hbox{\oalign{%
  $\scriptstyle\leftarrow$\cr
  \vrule width0pt height.5ex$\hfil\scriptstyle\relbar$\cr}}}}
\newcommand{\lrightarrows}{\mathrel{\raise.75ex\hbox{\oalign{%
  $\scriptstyle\relbar$\hfil\cr
  $\scriptstyle\vrule width0pt height.5ex\smash\rightarrow$\cr}}}}
\newcommand{\Rrelbar}{\mathrel{\raise.75ex\hbox{\oalign{%
  $\scriptstyle\relbar$\cr
  \vrule width0pt height.5ex$\scriptstyle\relbar$}}}}

\makeatletter
\def\leftrightarrowsfill@{\arrowfill@\leftrarrows\Rrelbar\lrightarrows}
\newcommand{\xleftrightarrows}[2][]{\ext@arrow 3399\leftrightarrowsfill@{#1}{#2}}
\makeatother

\newtheorem{thm}{Proposition}[section]
\newtheorem{Thm}[thm]{Theorem}
\newtheorem{rem}[thm]{Remark}

\newtheorem{cor}[thm]{Corollary}

\newtheorem{lem}[thm]{Lemma}
\newtheorem{defn}[thm]{Definition}

\newtheorem{examples}[thm]{Example}

\title{Categorification of modules and construction of schemes}
\author{Abhishek Banerjee \footnote{Department of Mathematics, Indian Institute of Science, Bangalore. Email: abhishekbanerjee1313@gmail.com.}
$\qquad$ Subhajit Das \footnote{Department of Mathematics, Indian Institute of Science, Bangalore. Email: subhajitdas@iisc.ac.in.} 
$\qquad$ Surjeet Kour \footnote{Department of Mathematics, Indian Institute of Technology, Delhi. Email: koursurjeet@gmail.com.}}
\date{}
\begin{document}

\linespread{0.95}
\maketitle 

\medskip

\begin{abstract} We use categorification of monoid actions to study algebraic geometry over symmetric monoidal categories. This brings together the relative algebraic geometry over symmetric monoidal categories  developed by To\"{e}n and Vaqui\'{e}, along with the theory of actegories over monoidal categories. We obtain schemes over a datum $(\mathcal C,\mathcal M)$, where $(\mathcal C,\otimes,1)$ is a symmetric monoidal category and 
$\mathcal M$ is an actegory over $\mathcal C$. One of our main tools is  using the datum 
$(\mathcal C,\mathcal M)$ to give a Grothendieck topology on the category  of affine schemes over $(\mathcal C,\otimes,1)$ that we call the ``spectral 
$\mathcal M$-topology.'' This consists of ``fpqc $\mathcal M$-coverings'' with certain special properties.   We provide a description of  schemes over 
$(\mathcal C,\mathcal M)$ in terms of quotients of disjoint unions of affine schemes over a certain equivalence relation.  These categories of schemes are closed under pullbacks and coproducts, and are equipped with change of base functors induced by symmetric monoidal adjuctions accompanied by lax $\mathcal C$-linear functors. 
\end{abstract}

\medskip
MSC(2020) Subject Classification: 14A30, 18M05


\section{Introduction}

In this paper, we use categorification of monoid actions to study algebraic geometry over symmetric monoidal categories. The relative algebraic geometry over a symmetric monoidal category $(\mathcal C,\otimes,1)$ was developed by To\"{e}n and Vaqui\'{e} \cite{TV}. Actegories over tensor categories have been studied in various contexts by a number of authors 
(see, for instance, B\'{e}nabou \cite{Ben}, Garner \cite{Garn}, Janelidze and Kelly \cite{JK}, Ostrik \cite{Ost}, Street \cite{Street}). We develop a notion of scheme over a datum $(\mathcal C,\mathcal M)$, where $\mathcal C$ is a symmetric monoidal category and $\mathcal M$ is an actegory over $\mathcal C$.

\smallskip
The idea of doing algebraic geometry over symmetric monoidal categories has roots in the work of Deligne \cite{Del}, as also in that of Hakim \cite{Hak}. In \cite{TV}, To\"{e}n and Vaqui\'{e}  gave a notion of scheme over a symmetric monoidal category $(\mathcal C, \otimes, 1)$ that is purely categorical, along with notions of Zariski immersions, Zariski coverings and fpqc coverings. By choosing the category  $(\mathcal C, \otimes, 1)$ appropriately, one opens up a number of new geometries by taking $\mathcal C$ for instance to be the category of sets, or commutative monoids, or the category of symmetric spectra. When $\mathcal C$ is taken to be the category of $\mathbb Z$-modules, one recovers the usual notion of a scheme over $Spec(\mathbb Z)$. There is also a homotopy version of the theory in \cite{TV}, obtained by taking $\mathcal C$ to be a symmetric monoidal model category. The latter has been developed into a theory of homotopical algebraic geometry using higher categorical structures by To\"{e}n and Vezzosi in \cite{TVe1}, \cite{TVe2}. The ideas of \cite{TV} have also been explored for instance in \cite{Ban1}, \cite{Ban2}, \cite{Ban3}, \cite{Mart1}, \cite{Mart2} (see also related work by Connes and Consani in \cite{CC}, 
\cite{CCJAG}, \cite{CCJNT}, 
\cite{CC2}). 

\smallskip
Let $(\mathcal C, \otimes, 1)$ be a symmetric monoidal category. In the framework of To\"{e}n and Vaqui\'{e}  \cite{TV}, the category of affine schemes is taken to be the opposite of that of commutative monoid objects in $(\mathcal C, \otimes, 1)$. Thereafter, the theory in \cite{TV} is developed by using objects in $\mathcal C$ equipped with module actions of commutative monoid objects in $(\mathcal C, \otimes, 1)$. This suggests that we should be able to study a similar algebraic geometry by categorifying not just the notion of a commutative monoid, but also that of modules over it. We see this as motivated by the ``microcosm principle'' (see  \cite[$\S$ 4.3]{BaDo}) which says that certain algebraic structures can be internalized in categorifications of the same strutures. For instance, since a (symmetric) monoidal category is the categorification of the notion of a (commutative) monoid, one studies (commutative) monoids in (symmetric) monoidal categories. Similarly,  actegories categorify monoid actions on sets and hence they are the natural setup for studying modules over commutative monoid objects in symmetric monoidal categories. 

\smallskip The starting point of this paper is to consider a datum $(\mathcal C,
\mathcal M)$, where $\mathcal C$ is a symmetric monoidal category and $\mathcal M$ is a category equipped with a functor 
\begin{equation}\label{1.1}
\boxtimes: \mathcal C\times\mathcal M\longrightarrow \mathcal M
\end{equation} satisfying certain conditions similar to that of an action of a monoid (see Section 2 for details). In recent years, there has been a lot of interest in such a framework $(\mathcal 
C,\mathcal M)$ (see, for instance, Etingof and Ostrik \cite{EO}, Garner \cite{Garn}, Gelaki \cite{Gel}, Walton and Yadav \cite{WaYa}).  The setup of actegories over monoidal categories has proved to be extremely versatile, with applications in various fields such as the theory of subfactors (see \cite{BEK}, 
\cite{O}), weak Hopf algebras (see \cite{Ost}, \cite{Ks}) and in conformal field theory 
(see \cite{FS}). This theory becomes particularly rich when the monoidal category has certain additional structures, such as rigidity or (multi-)fusion structures (see for instance, \cite{Et0}, \cite{Et}). As such, actegories over monoidal categories have increasingly become a topic of study in their own right. For more on this subject, we refer the reader for instance to 
\cite[$\S$ 5]{EGNO}. 

\smallskip
We start with a datum $(\mathcal C,\mathcal M)$ as in \eqref{1.1}.  Our construction is fairly general, especially in the sense that  we do not require the categories in \eqref{1.1}   to be additive. Given a commutative monoid object $a$ in $(\mathcal C,\otimes,1)$, we consider the category  $\mathbf{Mod}_{\mathcal M}(a)$ of $a$-modules in $\mathcal M$. The objects of  $\mathbf{Mod}_{\mathcal M}(a)$ are pairs $(m,\rho)$ where $m\in \mathcal M$ and $\rho:a\boxtimes m\longrightarrow m$ is a morphism satisfying conditions similar to a module action (see Definition \ref{D2.1dc}). We then build up the key notions of coverings by using properties of the adjoint functors of ``restriction of scalars'' and ``extension of scalars'' between these module categories. 

\smallskip One of the main  tools in our paper is the use of the datum 
$(\mathcal C,\mathcal M)$ to define a Grothendieck topology on the category $Aff_{\mathcal C}$ of affine schemes over $(\mathcal C,\otimes,1)$, which we refer to as the ``spectral 
$\mathcal M$-topology.'' This consists of ``fpqc $\mathcal M$-coverings'' with certain special properties that we explain in Section 3. We obtain schemes over $(\mathcal C,\mathcal M)$ using coverings by means of affine schemes in the category of sheaves for the spectral $\mathcal M$-topology on $Aff_{\mathcal C}$. We will typically refer to these as $\mathcal M$-schemes. The category of $\mathcal M$-schemes  appears to have good properties in that it is closed under pullbacks, disjoint unions and any scheme may be recovered as the quotient of an equivalence relation on the disjoint union of affine schemes.  
We also incorporate change of base functors between categories of schemes over respective data  $(\mathcal C,\mathcal M)$ and $(\mathcal C',\mathcal M')$ as in \eqref{1.1} connected by a lax functor and an adjunction between the symmetric monoidal categories
$\mathcal C$ and $\mathcal C'$. We mention here that in future, we hope to extend our methods to the homotopical context, where $\mathcal C$ and $\mathcal M$ additionally carry model structures. 

\smallskip
We now describe the paper in more detail. We begin with an actegory $\mathcal M$ over $\mathcal C$ as in \eqref{1.1}. Let $Comm(\mathcal C)$ denote the category of commutative monoid objects in $\mathcal C$. For $a\in Comm(\mathcal C)$, the category  $\mathbf{Mod}_{\mathcal M}(a)$ of $a$-modules with values in $\mathcal M$ is described in Section 2. For a morphism $\alpha: a\longrightarrow b$ in $Comm(\mathcal C)$, we construct an adjoint pair $(\alpha^\ast,\alpha_\ast)$ of extension and restriction of scalars between the categories $\mathbf{Mod}_{\mathcal M}(a)$ and  $\mathbf{Mod}_{\mathcal M}(b)$. Some of the material in Section 2 is possibly well known, but we record it here in order to fix notation. Further, we show that the assignment $a\mapsto  \mathbf{Mod}_{\mathcal M}(a)$ determines a pseudo-functor from $Comm(\mathcal C)$ to the $2$-category of categories. 

\smallskip
As in \cite{TV}, we take the category $Aff_{\mathcal C}$ of affine schemes to be the opposite of the category $Comm(\mathcal C)$ of commutative monoid objects in $\mathcal C$. For any $a\in Comm(\mathcal C)$, we let $Spec(a)\in Aff_{\mathcal C}=Comm(\mathcal C)^{op}$ denote the corresponding affine scheme. Similarly, for a morphism $\alpha:a\longrightarrow b$ in $Comm(\mathcal C)$, we denote by   $\alpha^{op}:
Spec(b)\longrightarrow Spec(a)$  the  corresponding morphism of affine schemes in $Aff_{\mathcal C}$.

\smallskip In Section 3, we define the topology on $Aff_{\mathcal C}$. We say that a morphism $\alpha:a\longrightarrow b$ in $Comm(\mathcal C)$ is $\mathcal M$-flat if $\alpha^*:\mathbf{Mod}_{\mathcal M}(a)\longrightarrow \mathbf{Mod}_{\mathcal M}(b)$ is exact. We will say that a collection $\{\alpha_i^{op}:Spec(a_i)\longrightarrow Spec(a)\}_{i\in I}$ of $\mathcal M$-flat morphisms is 
an fpqc $\mathcal M$-cover for $Spec(a)$ if there exists a finite subset $J\subseteq I$ such that the collection of functors
$
      \left\{  \alpha_j^{*}   : \mathbf{Mod}_{\mathcal M}(a) \longrightarrow  \mathbf{Mod}_{\mathcal M}(a_j)\right\}_{j\in J}
$
is conservative (see Definition \ref{fpqcM9}).  If $\alpha:a\longrightarrow b$ is $\mathcal M$-flat and an epimorphism of finite type in $Comm(\mathcal C)$, we will say that $\alpha^{op}$ is a spectral immersion relative to $\mathcal M$. A spectral $\mathcal M$-cover  $\{\alpha_i^{op}:Spec(a_i)\longrightarrow Spec(a)\}_{i\in I}$ is an fpqc $\mathcal M$-cover such that each $\alpha_i^{op}$ is a spectral  immersion relative to $\mathcal M$. The main result of Section 3 is that both the fpqc $\mathcal M$-coverings as well as the spectral $\mathcal M$-coverings define Grothendieck topologies on $Aff_{\mathcal C}$ (see Proposition 
\ref{checkfortopology}). From then onwards, we set 
$Sh(Aff_{\mathcal C})_{\mathcal M}$  be the category of sheaves on $Aff_{\mathcal C}$ with respect to the spectral $\mathcal M$-topology.

\smallskip  We say that the $\mathcal C$-actegory $(\mathcal M, \boxtimes)$ is subcanonical  if the fpqc $\mathcal M$-topology (and hence, the spectral $\mathcal M$-topology)   on $Aff_{\mathcal C}$ is subcanonical, i.e., for each $Spec(a)\in Aff_{\mathcal C}$, the representable presheaf $Aff_{\mathcal C}(\_\_, Spec(a))$ on $Aff_{\mathcal C}$ is a sheaf  for the fpqc $\mathcal M$-topology on $Aff_{\mathcal C}$.  In this situation, an affine scheme
$Spec(a)\in Aff_{\mathcal C}$ may be treated as an object of $Sh(Aff_{\mathcal C})_{\mathcal M}$. For most of this paper, we will assume that  $(\mathcal M, \boxtimes)$ is subcanonical. We mention that later on in Section 6, we have collected a number of examples of subcanonical actegories arising in various natural ways such as from diagonal actions, functor categories, representation categories of monoids and from directed graphs. 

\smallskip
In Section 4, we  give the definition of Zariski open immersion relative to $\mathcal M$ for a morphism $f:F\longrightarrow G$ in $Sh(Aff_{\mathcal C})_{\mathcal M}$ (see Definition \ref{zariskiopenforsheaves}). We show that this notion is stable under base change and closed under composition in 
$Sh(Aff_{\mathcal C})_{\mathcal M}$. By an $\mathcal M$-scheme (see Definition \ref{defin4.8}), we mean a sheaf $F\in Sh(Aff_{\mathcal C})_{\mathcal M}$ for the spectral 
$\mathcal M$-topology on $Aff_{\mathcal C}$ such that there is an epimorphism
\begin{equation}
\coprod_{i \in I} X_i \longrightarrow F
\end{equation}
in $Sh(Aff_{\mathcal C})_{\mathcal M}$  where each $X_i$ is an affine scheme and each $X_i\longrightarrow F$ is a Zariski open immersion relative to $\mathcal M$. We show that the full subcategory $Sch(\mathcal C)_{\mathcal M}$ of $Sh(Aff_{\mathcal C})_{\mathcal M}$ consisting of $\mathcal M$-schemes is closed under coproducts and pullbacks. We also provide a description (see Theorem \ref{equivalencerel}) of $\mathcal M$-schemes in terms of quotients of disjoint unions of affine schemes over an equivalence relation satisfying certain properties. 

\smallskip
In Section 5, we consider symmetric monoidal categories $(\mathcal C,\otimes,1)$ and $(\mathcal D,\otimes,1)$ which are connected by an adjunction $ \left(\mathbf B :  \mathcal C \longrightarrow \mathcal D,\text{ }\mathbf A : \mathcal D \longrightarrow \mathcal C\right)$ that satisfies certain monoidal properties.   We also consider  a  $\mathcal C$-actegory $(\mathcal M, \boxtimes)$, a $\mathcal D$-actegory $(\mathcal N, \boxplus)$ and a lax 
$\mathcal C$-linear functor 
\begin{equation}\label{1.3}
(\mathbb L,\Gamma):\mathbf B_\ast(\mathcal N)=(\mathcal N,\boxplus^{\mathbf B})\longrightarrow (\mathcal M,\boxtimes)
\end{equation} where $\boxplus^{\mathbf B}$ denotes the restriction of the $\mathcal D$-action on $\mathcal N$ to a $\mathcal C$-action by means of the functor $\mathbf B :  \mathcal C \longrightarrow \mathcal D$. In that  case,  we have an adjoint pair  (see Theorem \ref{basechangingscheme})
\begin{equation} (\widetilde{\mathbf A_!} : Sh(Aff_{\mathcal C})_{\mathcal M} \longrightarrow Sh(Aff_{\mathcal D})_{\mathcal N}, \qquad \widetilde{\mathbf B_!} : Sh(Aff_{\mathcal D})_{\mathcal N} \longrightarrow Sh(Aff_{\mathcal C})_{\mathcal M})
\end{equation} 
    Additionally, if $\mathcal M$ and $\mathcal N$ are subcanonical, then we obtain a change of base functor  $\widetilde{\mathbf A_!}: Sch(\mathcal C)_{\mathcal M} \longrightarrow Sch(\mathcal D)_{\mathcal N}$ at the level of schemes.   

\smallskip
{\bf Acknowledgements:} Authors A.B. and S.K. are grateful to the Research in Pairs program at the CIRM in Marseille, where part of this paper 
was written.  Author S.D. was supported by Prime Minister's Research Fellowship PMRF-21-4890.03.

\section{Categorical actions and the pseudo-functor $\mathbf{Mod}_{\mathcal M}$}
Throughout, we let $(\mathcal C, \otimes, 1)$ be a closed symmetric monoidal category which is both complete and cocomplete. We recall (see \cite{act}, \cite{JK}) that a left $\mathcal C$-actegory  $(\mathcal M, \boxtimes, \lambda, l)$  consists of a bifunctor $\boxtimes : \mathcal C \times \mathcal M \longrightarrow \mathcal M$ and a collection of isomorphisms
\begin{equation}
    \begin{split}
    \lambda &= \left(\lambda_{a, b, m} : a \boxtimes (b \boxtimes m) \xrightarrow{\sim} (a \otimes b) \boxtimes m \big{\vert} a, b \in \mathcal C, m \in \mathcal M\right)\\
    l &= \left(l_m : 1 \boxtimes m \xrightarrow{\sim} m \big{\vert} m \in \mathcal M\right)
    \end{split}
\end{equation} which are natural and subject to certain coherence axioms. Right $\mathcal C$-actegories are defined similarly. Let $\mathcal M = (\mathcal M, \boxtimes)$ and $\mathcal M' = (\mathcal M', \boxtimes')$ be left $\mathcal C$-actegories. A lax $\mathcal C$-linear functor $(L, J) : \mathcal M \longrightarrow \mathcal M'$ consists of a functor $L : \mathcal M \longrightarrow \mathcal M'$ together with a natural transformation $J = \{J_{c, m} : c \boxtimes' L(m) \longrightarrow L(c \boxtimes m) : (c, m) \in \mathcal C \times \mathcal M\}$ satisfying certain coherence axioms (see \cite{act}, \cite{JK}). Additionally, if $J$ is a natural isomorphism, then $(L, J)$ is called a strong $\mathcal C$-linear functor. We denote the category of left $\mathcal C$-actegories and lax $\mathcal C$-linear functors by $\mathcal C-Act^{lx}$.

\smallskip
From now onwards, we fix a left $\mathcal C$-actegory $(\mathcal M, \boxtimes, \lambda, l)$ such that the underlying category $\mathcal M$ is both complete and cocomplete. Additionally, for every $a \in \mathcal C, m \in \mathcal M$, we assume that the functors $a \boxtimes \_\_ : \mathcal M \longrightarrow \mathcal M$ and $\_\_ \boxtimes m : \mathcal C \longrightarrow \mathcal M$ preserve colimits. We let $Comm(\mathcal C)$ denote the category of commutative monoid objects in $\mathcal C$.  

\begin{defn}\label{D2.1dc}
   Let $a = (a, \mu : a \otimes a \longrightarrow a, \iota : 1 \longrightarrow a) \in Comm(\mathcal C)$. A left $a$-module in $\mathcal M$ is a pair $(m, \rho)$ consisting of an object $m \in M$ and a morphism $\rho : a \boxtimes m \longrightarrow m$ such that
    \begin{equation}\label{newnameoldstuff}
        \rho \circ (\mu \boxtimes 1_m) \circ \lambda_{a, a, m} = \rho \circ (1_a \boxtimes \rho):a\boxtimes (a\boxtimes m)\longrightarrow m\qquad \qquad \rho \circ (\iota \boxtimes 1_m) = l_m: 1\boxtimes m\longrightarrow m
    \end{equation}
    A morphism $\psi : (m, \rho) \longrightarrow (m', \rho')$ of left $a$-modules is a morphism $\psi : m \longrightarrow m'$ in $\mathcal M$ such that $\rho' \circ (1_a \boxtimes \psi) = \psi \circ \rho$. For $a \in Comm(\mathcal C)$, we let $\mathbf{Mod}_{\mathcal M}(a)$ denote the category of left $a$-modules. 
\end{defn}

\begin{rem}\label{creation}
\emph{For $a \in Comm(\mathcal C)$, we may consider the associated monad $\mathbf T = a \boxtimes \_\_$ on $\mathcal M$. Then, the Eilenberg-Moore category of $\mathbf T$-algebras is the category $\mathbf{Mod}_{\mathcal M}(a)$. Hence, the forgetful functor $U_a :  \mathbf{Mod}_{\mathcal M}(a) \longrightarrow \mathcal M$ reflects isomorphisms and creates limits. Further, since the endofunctor underlying $\mathbf T$ preserves colimits, $U_a$ also creates colimits. In particular, since $\mathcal M$ is complete and cocomplete, it follows that $\mathbf{Mod}_{\mathcal M}(a)$ is complete and cocomplete and $U_a$ preserves both limits and colimits.}
\end{rem}
Let $\alpha : (a, \mu_a, \iota_a) \longrightarrow (b, \mu_b, \iota_b)$ be a morphism in $Comm(\mathcal C)$. We consider the restriction of scalars   $\alpha_{*} : \mathbf{Mod}_{\mathcal M}(b) \longrightarrow \mathbf{Mod}_{\mathcal M}(a)$, given by $(m, \rho) \mapsto (m, \rho \circ (\alpha \boxtimes 1_m))$.  The extension of scalars $\alpha^*:  \mathbf{Mod}_{\mathcal M}(a) \longrightarrow \mathbf{Mod}_{\mathcal M}(b)$ may now be defined as follows: for 
      $(m, \rho) \in \mathbf{Mod}_{\mathcal M}(a)$, we consider the object
     \begin{equation}\label{coeq}
       b \boxtimes_a m := Coeq\left(
\begin{tikzcd}
b \boxtimes (a \boxtimes m) \arrow[rrrrr, "1_b \boxtimes \rho", shift left=4] \arrow[rr, "\lambda_{b,a,m}"', shift right] &  & (b \otimes a) \boxtimes m \arrow[rrr, "(\mu_b \circ (1_b \otimes \alpha)) \boxtimes 1_m"', shift right] &  &  & b \boxtimes m
\end{tikzcd}\right)
     \end{equation} in $\mathcal M$. We  denote the canonical epimorphism $b \boxtimes m \longrightarrow b \boxtimes_a m$ by $coeq_{\alpha, (m, \rho)}$.  Using Remark $\ref{creation}$, it follows that \eqref{coeq} is also a  coequalizer in $\mathbf{Mod}_{\mathcal M}(b)$.  
\begin{lem}\label{extension}
 There is  a canonical morphism $\widetilde{\rho} : b \boxtimes (b \boxtimes_a m) \longrightarrow b \boxtimes_a m$ making $b \boxtimes_a m$ into a left $b$-module.
\end{lem}
\begin{proof}
We consider the following diagram :
\begin{equation}\label{bigdiag}\small
\begin{tikzcd}
b \boxtimes (b \boxtimes (a \boxtimes m)) \arrow[rr, "{1_b \boxtimes \lambda_{b, a, m}}"', shift right] \arrow[rrrrr, "1_b \boxtimes (1_b \boxtimes \rho)", shift left=4] \arrow[dd, "{\lambda_{b, b, a \boxtimes m}}"'] &  & b \boxtimes ((b \otimes a) \boxtimes m) \arrow[rrr, "1_b \boxtimes ((\mu_b \circ (1_b \otimes \alpha)) \boxtimes 1_m)"', shift right] &                                                                                                                                 &  & b \boxtimes (b \boxtimes m) \arrow[r, "1_b \boxtimes coeq"] \arrow[dd, "{\lambda_{b, b, m}}"'] & b \boxtimes (b \boxtimes_a m) \arrow[dd, "1_{b \boxtimes (b \boxtimes_a m)}"'] \\
                                                                                                                                                                                                                       &  &                                                                                                                                  &                                                                                                                                 &  &                                                                         &                                                                                \\
(b \otimes b) \boxtimes (a \boxtimes m) \arrow[rrrrr, "1_{b \otimes b} \boxtimes \rho", shift left=3] \arrow[rr, "{\lambda_{b \otimes b, a, m}}"', shift right] \arrow[dd, "\mu_b \boxtimes 1_{a \boxtimes m}"']        &  & ((b \otimes b) \otimes a) \boxtimes m \arrow[r, "\sim"', shift right]                                                            & (b \otimes (b \otimes a)) \boxtimes m \arrow[rr, "(1_b \otimes (\mu_b \circ (1_b \otimes \alpha)))\boxtimes 1_m"', shift right] &  & (b \otimes b) \boxtimes m \arrow[r] \arrow[dd, "\mu_b \boxtimes 1_m"']  & b \boxtimes (b \boxtimes_a m) \arrow[dd, dashed]                               \\
                                                                                                                                                                                                                       &  &                                                                                                                                  &                                                                                                                                 &  &                                                                         &                                                                                \\
b \boxtimes (a \boxtimes m) \arrow[rrrrr, "1_b \boxtimes \rho", shift left=3] \arrow[rr, "{\lambda_{b, a, m}}"', shift right]                                                                                           &  & (b \otimes a) \boxtimes m \arrow[rrr, "(\mu_b \circ (1_b \otimes \alpha)) \boxtimes 1_m"', shift right]                               &                                                                                                                                 &  & b \boxtimes m \arrow[r, "coeq"]                                                 & b \boxtimes_a m                                                               
\end{tikzcd}
\end{equation}
Since $b \boxtimes \_\_$ preserves colimits, it is evident that the top row of $\eqref{bigdiag}$ is a coequalizer. 

\smallskip
(1) Because of the naturality of the isomorphisms in $\lambda$,   the associativity constraints $\lambda_{b, b, a \boxtimes m}$ and $\lambda_{b, b, m}$  satisfy $\lambda_{b, b, m}\circ (1_b \boxtimes (1_b \boxtimes \rho))=(1_{b \otimes b} \boxtimes \rho)\circ \lambda_{b, b, a \boxtimes m}$.

\smallskip
(2)  Using the pentagon axiom in $\mathcal M$ and the naturality of $\lambda$, we see that the 
    $\lambda_{b, b, a \boxtimes m}$ and $\lambda_{b, b, m}$ fit into a commutative square  
    \begin{equation*}
    \xymatrix{b \boxtimes (b \boxtimes (a \boxtimes m))\ar[rr]^{ 1_b \boxtimes \lambda_{b, a, m}} 
    \ar[d]^{\lambda_{b, b, a \boxtimes m}}&&b \boxtimes ((b \otimes a) \boxtimes m) \ar[rrrr]^{1_b \boxtimes ((\mu_b \circ (1_b \otimes \alpha)) \boxtimes 1_m)} &&& &  b \boxtimes (b \boxtimes m)\ar[d]_{\lambda_{b, b, m}}\\
    (b \otimes b) \boxtimes (a \boxtimes m) \ar[rr]_{\lambda_{b \otimes b, a, m}} &&((b \otimes b) \otimes a) \boxtimes m\ar[r]^{\sim}&(b \otimes (b \otimes a)) \boxtimes m
    \ar[rrr]_{\quad (1_b \otimes (\mu_b \circ (1_b \otimes \alpha)))\boxtimes 1_m}&& &  (b \otimes b) \boxtimes m \\
    }
    \end{equation*}

\smallskip
(3) Using the functoriality of $\boxtimes$, we see that  $(1_b \boxtimes \rho)\circ (\mu_b \boxtimes 1_{a \boxtimes m})=(\mu_b \boxtimes 1_m)\circ (1_{b \otimes b} \boxtimes \rho)$.

\smallskip
(4)  Using the naturality of $\lambda$, the pentagon axiom in $\mathcal C$, the associativity of $\mu_b$ and the functoriality of $\boxtimes$, we have the commutative diagram
 \begin{equation*}
    \xymatrix{
  (b \otimes b) \boxtimes (a \boxtimes m) \ar[d]_{\mu_b\boxtimes 1_{a\boxtimes m}} \ar[rr]^{\lambda_{b \otimes b, a, m}} &&((b \otimes b) \otimes a) \boxtimes m\ar[r]^{\sim}&(b \otimes (b \otimes a)) \boxtimes m
    \ar[rrr]^{\quad (1_b \otimes (\mu_b \circ (1_b \otimes \alpha)))\boxtimes 1_m}&& &  (b \otimes b) \boxtimes m \ar[d]^{\mu_b\boxtimes 1_m} \\
  b \boxtimes (a \boxtimes m) \ar[rr]^{\lambda_{b, a, m}} 
   && (b \otimes a) \boxtimes m\ar[rrrr]^{(\mu_b \circ (1_b \boxtimes \alpha)) \boxtimes 1_m} &&& &    (b \boxtimes m) \\
    }
    \end{equation*}

\smallskip
Using (1) and (2), it follows that the diagrams in the top and middle rows of $\ref{bigdiag}$ are naturally isomorphic. Hence, the coequalizer of the middle row is $(b \otimes b) \boxtimes m \xrightarrow{(1_b \boxtimes coeq) \circ \lambda_{b, b, n}^{-1}} b \boxtimes (b \boxtimes_a m)$.

\smallskip
Using (3) and (4), there is an induced morphism $\widetilde{\rho} : b \boxtimes (b \boxtimes_a m) \longrightarrow b \boxtimes_a m$. It may be verified that $\widetilde{\rho}$ makes $b \boxtimes_a m$ into a left $b$-module.
\end{proof}
\begin{Thm}\label{T2.2xs} Let $\alpha : (a, \mu_a, \iota_a) \longrightarrow (b, \mu_b, \iota_b)$  be a morphism in $Comm(\mathcal C)$.
    Then, the assignment $(m, \rho) \mapsto (b \boxtimes_a m, \widetilde{\rho})$ extends to a functor $\alpha^* : \mathbf{Mod}_{\mathcal M}(a) \longrightarrow \mathbf{Mod}_{\mathcal M}(b)$ which is left adjoint to the restriction of scalars $\alpha_{*} : \mathbf{Mod}_{\mathcal M}(b) \longrightarrow \mathbf{Mod}_{\mathcal M}(a)$. 
\end{Thm}
\begin{proof}
    Let $(m, \rho) \in \mathbf{Mod}_{\mathcal M}(a)$, $(n,\tau)\in \mathbf{Mod}_{\mathcal M}(b)$. We need to show that there is a natural isomorphism 
  \begin{equation}\label{Eqth2.5} 
    \mathbf{Mod}_{\mathcal M}(b)\left((b \boxtimes_a m, \widetilde{\rho}), (n, \tau)\right) \cong \mathbf{Mod}_{\mathcal M}(a)\left((m, \rho), \alpha_{*}((n, \tau))\right)
  \end{equation}
  Let $\phi : (b \boxtimes_a m, \widetilde{\rho}) \longrightarrow (n, \tau)$ be a morphism in $\mathbf{Mod}_{\mathcal M}(b)$. We consider  $\overline{\phi} : m \longrightarrow n$ given by the composite
    \begin{equation}
        m \cong 1 \boxtimes m \xrightarrow{\iota_b \boxtimes 1_m} b \boxtimes m \xrightarrow{\quad coeq_{\alpha, (m, \rho)}\quad} b \boxtimes_a m \xrightarrow{\quad\phi\quad} n
    \end{equation} It is easy to see that $\overline{\phi} \in \mathbf{Mod}_{\mathcal M}(a)((m, \rho), \alpha_{*}((n, \tau)))$.
    We check the naturality of the collection of maps
    \begin{equation}\label{collec} 
     \left\{\mathbf{Mod}_{\mathcal M}(b)\left((b \boxtimes_a m, \widetilde{\rho}), (n, \tau)\right) \longrightarrow \mathbf{Mod}_{\mathcal M}(a)\left((m, \rho), \alpha_{*}((n, \tau))\right),\quad \phi \mapsto \overline{\phi}\quad\big{\vert}\quad (n, \tau) \in \mathbf{Mod}_{\mathcal M}(b) \right\}
     \end{equation}Let $\xi : (n, \tau) \longrightarrow (n', \tau')$ be any morphism  in $\mathbf{Mod}_{\mathcal M}(b)$. For any $\phi \in \mathbf{Mod}_{\mathcal M}(b)\left((b \boxtimes_a m, \widetilde{\rho}), (n, \tau)\right)$, we have
     \begin{equation}
         \overline{\xi \circ \phi} = \left(\begin{tikzcd}
m \arrow[r, "\sim"] & 1 \boxtimes m \arrow[r, "\iota_b \boxtimes 1_m"] & b \boxtimes m \arrow[r, "coeq"] & b \boxtimes_a m \arrow[r, "\phi"] & n \arrow[r, "\xi"] & n'
\end{tikzcd}\right)= \xi \circ \overline{\phi}. 
     \end{equation}
     This proves the naturality of the collection $\eqref{collec}$.
     
    \smallskip
    Now let $\psi : (m, \rho) \longrightarrow \alpha_{*}((n, \tau))$ be a morphism in $\mathbf{Mod}_{\mathcal M}(a)$. We note that $\tau \circ (\alpha \boxtimes \psi) = \psi \circ \rho$. Hence,
    \begin{equation}\label{eq2.5}
      \begin{split}
          \tau \circ (1_b \boxtimes \psi) \circ (1_b \boxtimes \rho) &= 
          \tau \circ (1_b \boxtimes (\psi \circ \rho))\\
          &= \tau \circ (1_b \boxtimes (\tau \circ (\alpha \boxtimes \psi)))\\
          &= \tau \circ (1_b \boxtimes \tau) \circ (1_b \boxtimes (\alpha \boxtimes \psi))\\
          &= \tau \circ (\mu_b \boxtimes 1_n) \circ \lambda_{b, b, n} \circ (1_b \boxtimes (\alpha \boxtimes \psi))\qquad[\text{the associativity of }\tau]\\
          &= \tau \circ (\mu_b \boxtimes 1_n) \circ ((1_b \otimes \alpha) \boxtimes \psi) \circ \lambda_{b, a, m}\qquad[\text{the naturality of }\lambda]\\
          &= \tau \circ ((\mu_b \circ (1_b \otimes \alpha)) \boxtimes \psi) \circ \lambda_{b, a, m}\\
          &= \tau \circ (1_b \boxtimes \psi) \circ ((\mu_b \circ (1_b \otimes \alpha)) \boxtimes 1_m) \circ \lambda_{b, a, m}
      \end{split}
    \end{equation}
    By the universal property of the coequalizer $\eqref{coeq}$, there exists a unique morphism $\underline{\psi} : b \boxtimes_a m \longrightarrow n$ such that the triangle
\begin{equation}\label{eq2.7}
 \begin{tikzcd}
b \boxtimes m \arrow[rr, "coeq_{\alpha, (m, \rho)}"] \arrow[rd, "\tau \circ (1_b \boxtimes \psi)"'] &   & b \boxtimes_a m \arrow[ld, "\underline{\psi}", dashed] \\
                                                                                   & n &                                            
\end{tikzcd}
\end{equation}commutes. To show that $\underline\psi$ is a morphism in $ \mathbf{Mod}_{\mathcal M}(b)$, we consider the diagram
\begin{equation}\label{Eqth2.11}
\begin{tikzcd}
(b \otimes b) \boxtimes m \arrow[rrrr, "{(1_b \boxtimes coeq_{\alpha, (m, \rho)}) \circ \lambda_{b, b, m}^{-1}}"] \arrow[d, "\mu_b \boxtimes 1_m"'] &  &  &  & b \boxtimes (b \boxtimes_a m) \arrow[d, "\widetilde{\rho}"'] \arrow[rr, "1_b \boxtimes \underline{\psi}"] &  & b \boxtimes n \arrow[d, "\tau"] \\
b \boxtimes m \arrow[rrrr, "{coeq_{\alpha, (m, \rho)}}"]                                                                                            &  &  &  & b \boxtimes_a m \arrow[rr, "\underline{\psi}"]                                                            &  & n                              
\end{tikzcd}
\end{equation}It follows from the diagram $\eqref{bigdiag}$ that the left hand square in \eqref{Eqth2.11} commutes. 
 Using $\eqref{eq2.7}$ and the naturality of $\lambda$, it may be verified that
$
    \tau \circ (1_b \boxtimes \underline{\psi}) \circ (1_b \boxtimes coeq_{\alpha, (m, \rho)}) \circ \lambda_{b, b, m}^{-1} = \underline{\psi} \circ coeq_{\alpha, (m, \rho)} \circ (\mu_b \boxtimes 1_m)
$, which shows that the outer square in \eqref{Eqth2.11} commutes. We now note that the morphism 
\begin{equation} 
(1_b \boxtimes coeq_{\alpha, (m, \rho)}) \circ \lambda_{b, b, m}^{-1} : (b \otimes b) \boxtimes m \longrightarrow b \boxtimes (b \boxtimes_a m)
\end{equation} is a coequalizer and hence an epimorphism. It now follows from the diagram \eqref{Eqth2.11} that $\tau \circ (1_b \boxtimes \underline{\psi}) = \underline{\psi} \circ \widetilde{\rho}$. This shows that $\underline{\psi} : (b \boxtimes_a m, \widetilde{\rho}) \longrightarrow (n, \tau)$ is a morphism in $\mathbf{Mod}_{\mathcal M}(b)$.

\smallskip
It is straightforward to check that the two assignments $\phi \mapsto \overline{\phi}$ and $\psi \mapsto \underline{\psi}$ are mutual inverses.  
It follows by \cite[Proposition 6.7.2]{Bor2} that there is a unique functor $\alpha^* : \mathbf{Mod}_{\mathcal M}(a) \longrightarrow  \mathbf{Mod}_{\mathcal M}(b)$ which is left adjoint to $\alpha_{*}$ and such that $\alpha^*((m, \rho)) = (b \boxtimes_a m, \widetilde{\rho})$ for every $(m ,\rho) \in \mathbf{Mod}_{\mathcal M}(a)$. This completes the proof.
\end{proof}
It may be verified that for a morphism $\phi : (m, \rho) \longrightarrow (m', \rho')$  in $\mathbf{Mod}_{\mathcal M}(a)$, the induced morphism $1_b \boxtimes_a \phi := \alpha^*(\phi)$  satisfies $coeq_{\alpha, (m', \rho')}\circ (1_b \boxtimes \phi)=(1_b \boxtimes_a \phi)\circ coeq_{\alpha, (m, \rho)}$. 
 We shall denote the unit and counit of the adjunction $(\alpha^* ,\alpha_*)$ by $\eta_{\alpha}$ and $\varepsilon_{\alpha}$ respectively.
 
 \smallskip 
Let the following be  a pushout square in $Comm(\mathcal C)$. 
\begin{equation}\label{cocart} 
\begin{tikzcd}
b'                     & b \arrow[l, "\alpha'"']                   \\
a' \arrow[u, "\beta'"] & a \arrow[l, "\alpha"] \arrow[u, "\beta"']
\end{tikzcd}
\end{equation}  We recall (see for instance, \cite[$\S$ 2]{TV}) that the pushout $b'$ is given by $b \otimes_a a'$.
\begin{lem}\label{associativity}
    Let $\alpha : a \longrightarrow a', \beta : a \longrightarrow b$ and $\gamma : c \longrightarrow b$ be morphisms in $Comm(\mathcal C)$. Then,

    \smallskip
    (1) $(a' \otimes_a b) \boxtimes m \cong a' \boxtimes_a (b \boxtimes m)$ as left $a'$-modules, and this is natural in $m \in \mathcal M$.

    \smallskip
    (2) $(a' \otimes_a b) \boxtimes_c m \cong a' \boxtimes_a (b \boxtimes_c m)$ as left $a'$-modules, and this is natural in $(m, \rho) \in \mathbf{Mod}_{\mathcal M}(c)$.
       
\end{lem}
\begin{proof}
(1) Since $\_\_ \boxtimes m : \mathcal{C} \longrightarrow \mathcal M$ preserves colimits, we have 
\begin{equation}\label{diag2.15}\footnotesize
(a' \otimes_a b) \boxtimes m \cong Coeq\left(
\begin{tikzcd}
a' \boxtimes (a \boxtimes (b \boxtimes m)) \cong (a' \otimes a \otimes b) \boxtimes m \arrow[rrr, "(1_{a'} \otimes (\mu_b \circ (\beta \otimes 1_b))) \boxtimes 1_m", shift left=2] \arrow[rrr, "((\mu_{a'} \circ (1_{a'} \otimes \alpha)) \otimes 1_b) \boxtimes 1_m"', shift right=2] &  &  & (a' \otimes b) \boxtimes m \cong a' \boxtimes (b \boxtimes m)
\end{tikzcd}\right) \cong a' \boxtimes_a (b \boxtimes m)
\end{equation} in $\mathcal M$. Using Remark \ref{creation}, it may be verified that $\eqref{diag2.15}$ holds in $\mathbf{Mod}_{\mathcal M}(a')$.

\smallskip
(2) We consider the diagram
\begin{equation}\label{diag2.16}
\begin{tikzcd}
b \boxtimes (c \boxtimes m) \arrow[rrrrr, "1_b \boxtimes \rho", shift left=2] \arrow[rrrrr, "{((\mu_b \circ (1_b \otimes \gamma)) \boxtimes 1_m) \circ \lambda_{b, c, m}}"', shift right=2] &  &  &  &  & b \boxtimes m
\end{tikzcd}
\end{equation}in $\mathbf{Mod}_{\mathcal M}(b)$. Applying $\beta_*$ to \eqref{diag2.16}, we obtain a diagram in $\mathbf{Mod}_{\mathcal M}(a)$. Using Remark $\ref{creation}$, it follows that the coequalizer of this diagram in $\mathbf{Mod}_{\mathcal M}(a)$ is $b \boxtimes m \xrightarrow{\text{ }coeq_{\gamma, (m, \rho)}\text{ }} b \boxtimes_c m$. Since $\alpha^* : \mathbf{Mod}_{\mathcal M}(a) \longrightarrow \mathbf{Mod}_{\mathcal M}(a')$ is a left adjoint, it preserves coequalizers. Hence,
\begin{equation}
\begin{tikzcd}
a' \boxtimes_a (b \boxtimes (c \boxtimes m)) \arrow[rrrrr, "1_{a'} \boxtimes_a (1_b \boxtimes \rho)", shift left=2] \arrow[rrrrr, "{{1_{a'} \boxtimes (((\mu_b \circ (1_b \otimes \gamma)) \boxtimes 1_m) \circ \lambda_{b, c, m})}}"', shift right=2] &  &  &  &  &  a' \boxtimes_a (b \boxtimes m) \arrow[rrr, "{1_{a'} \boxtimes coeq_{\gamma, (m, \rho)}}", dashed] &  &  & a' \boxtimes_a (b \boxtimes_c m)
\end{tikzcd}
\end{equation}is a coequalizer in $\mathbf{Mod}_{\mathcal M}(a')$. Using part (1), it follows that $(a' \otimes_a b) \boxtimes_c m \cong a' \boxtimes_a (b \boxtimes_c m)$ as left $a'$-modules.
\end{proof}

We denote by $Cat$ the $2$-category of categories (see, for instance, \cite[$\S$ 7]{Bor}). The following result will be used in later sections.
\begin{thm}\label{pseudo}
(1) The assignment\begin{equation}
  \begin{array}{c}
  \mathbf{Mod}_{\mathcal M} : Comm(\mathcal C) \longrightarrow Cat\\
\left(  Comm(\mathcal C)\ni a \mapsto \mathbf{Mod}_{\mathcal M}(a)\right) \qquad \left((\alpha:a\longrightarrow b) \mapsto (\alpha^{*}:\mathbf{Mod}_{\mathcal M}(a)\longrightarrow \mathbf{Mod}_{\mathcal M}(b))\right)
  \end{array}
\end{equation}
is a pseudo-functor.  

\smallskip
(2) For every  $\alpha : a \longrightarrow b$ in $Comm(\mathcal C)$, the restriction of scalars $\alpha_{*} : \mathbf{Mod}_{\mathcal M}(b) \longrightarrow \mathbf{Mod}_{\mathcal M}(a)$ is conservative, i.e., a morphism $\phi$ in $\mathbf{Mod}_{\mathcal M}(b)$ is an isomorphism if and only if $\alpha_*(\phi)$ is an isomorphism in $\mathbf{Mod}_{\mathcal M}(a)$.

\smallskip
(3) For any pushout square $\eqref{cocart}$ in $Comm(\mathcal C)$, there is a  natural isomorphism $\alpha^* \circ \beta_* \overset{\sim}{\longrightarrow} \beta'_* \circ \alpha'^*$.  
\end{thm}
\begin{proof}
(1) It is easy to see that for any $a \in Comm(\mathcal C)$, $(1_a)^* \cong 1_{\mathbf{Mod}_{\mathcal M}(a)}$. Now, let $a \xrightarrow{\text{ }\alpha\text{ }} b \xrightarrow{\text{ }\beta\text{ }} c$ be morphisms in $Comm(\mathcal C)$. We  consider $\beta : b \longrightarrow c, 1_b : b \longrightarrow b$ and $\alpha : a \longrightarrow b$. Using Lemma $\ref{associativity}$, it follows that 
\begin{equation}
    (\beta \circ \alpha)^* \cong c \boxtimes_a \_\_ \cong (c \otimes_b b) \boxtimes_a \_\_ \cong c \boxtimes_b (b \boxtimes_a \_\_) \cong \beta^* \circ \alpha^*
\end{equation}It can be checked that $\mathbf{Mod}_{\mathcal M}$ satisfies the coherence axioms for a pseudo-functor.

\smallskip
(2) This follows from the fact that for every $c \in \mathcal C$, the forgetful functor $\mathbf{Mod}_{\mathcal M}(c) \longrightarrow \mathcal M$ reflects isomorphisms.

\smallskip
(3)  Applying Lemma $\ref{associativity}$ to the morphisms $\alpha : a \longrightarrow a', \beta : a \longrightarrow b, 1_b : b \longrightarrow b$, it follows that
\begin{equation}
    \beta'_{*} \circ \alpha'^* = b' \boxtimes_b \_\_ = (a' \otimes_a b) \boxtimes_b \_\_ \cong a' \boxtimes_a (b \boxtimes_b \_\_) \cong a' \boxtimes_a \_\_ = \alpha^* \circ \beta_*
\end{equation} 
\end{proof}

\section{The topology on $Aff_{\mathcal C}$}
We set $Aff_{\mathcal C} := Comm(\mathcal C)^{op}$. For any object $a \in Comm(\mathcal C)$, the corresponding object in $Aff_{\mathcal C}$ will be denoted by $Spec(a)$. For any morphism $\alpha : a \longrightarrow b$ in $Comm(\mathcal C)$, the corresponding morphism in $Aff_{\mathcal C}$ will be denoted by $\alpha^{op} : Spec(b) \longrightarrow Spec(a)$. In this section, we will introduce the topologies on $Aff_{\mathcal C}$ relative to $\mathcal M$. 
\begin{defn}
    A morphism $\alpha : a \longrightarrow b$ in $Comm(\mathcal C)$ is $\mathcal M$-flat if the extension of scalars $\alpha^* : \mathbf{Mod}_{\mathcal M}(a) \longrightarrow \mathbf{Mod}_{\mathcal M}(b)$ preserves finite limits.
\end{defn}
\begin{lem}\label{flat}
$\mathcal M$-flatness is stable under pushouts and closed under compositions in $Comm(\mathcal C)$.
\end{lem}
\begin{proof}  Using Proposition \ref{pseudo}, it is straightforward to check that the composition of $\mathcal M$-flat morphisms is  $\mathcal M$-flat. We  now consider the following pushout square in $Comm(\mathcal C)$  such that $\alpha$ is $\mathcal M$-flat. 
\begin{equation}\label{cocart7y} 
\begin{tikzcd}
b'     = b \otimes_a a'                & b \arrow[l, "\alpha'"']                   \\
a' \arrow[u, "\beta'"] & a \arrow[l, "\alpha"] \arrow[u, "\beta"']
\end{tikzcd}
\end{equation} We need to show that $\alpha'$ is $\mathcal M$-flat, i.e., $\alpha'^* : \mathbf{Mod}_{\mathcal M}(b) \longrightarrow \mathbf{Mod}_{\mathcal M}(b \otimes_a a')$ preserves finite limits. Let $\mathbb D : \mathcal I \longrightarrow \mathbf{Mod}_{\mathcal M}(b)$ be a diagram, where $\mathcal I$ is a finite category. There is a canonical morphism
    \begin{equation}\label{equation3.1}
         \alpha'^*\left(\underset{\mathcal I}{lim}\mathbb D\right) \longrightarrow \underset{\mathcal I}{lim} \left(\alpha'^* \circ \mathbb D\right)
    \end{equation}in $\mathbf{Mod}_{\mathcal M}(b \otimes_a a')$. Applying $\beta'_*$  to $\eqref{equation3.1}$, using Proposition $\ref{pseudo}$ and the fact that $\beta_*$ and $\beta'_*$ preserve limits, we have a morphism
    \smaller
    \begin{equation}\label{equation3.2}
       \alpha^*\left(\underset{\mathcal I}{lim} \left(\beta_* \circ \mathbb D\right)\right) \cong \alpha^*\left(\beta_*\left(\underset{\mathcal I}{lim} \mathbb D\right)\right) \cong \beta'_*\left(\alpha'^*\left(\underset{\mathcal I}{lim} D\right)\right) \longrightarrow \beta'_*\left(\underset{\mathcal I}{lim} \left(\alpha'^* \circ \mathbb D\right)\right) \cong \underset{\mathcal I}{lim} \left(\beta'_* \circ \alpha'^* \circ \mathbb D\right) \cong \underset{\mathcal I}{lim} \left(\alpha^* \circ (\beta_* \circ \mathbb D)\right)
    \end{equation}
    \normalsize
    Since $\alpha$ is $\mathcal M$-flat, the morphism in $\eqref{equation3.2}$ must be an isomorphism. Further since $\beta'_*$ is conservative, it follows that the morphism in $\eqref{equation3.1}$ is an isomorphism. Hence, $\alpha'$ is $\mathcal M$-flat.
\end{proof}
For an object $a \in Comm(\mathcal C)$, we consider the coslice category $a/Comm(\mathcal C)$, or the category of ``$a$-algebras.''  We recall that objects of $a/Comm(\mathcal C)$ are the morphisms in $Comm(\mathcal C)$ with domain $a$.  Let $\alpha : a \longrightarrow b$ be a morphism in $Comm(\mathcal C)$. We also recall (see \cite[$\S$ 2.3]{TV}) that $\alpha$ is of finite type if for any filtered diagram $\mathbb D : \mathcal I \longrightarrow a/Comm(\mathcal C)$, the canonical map
\begin{equation}
    \underset{i \in \mathcal I}{colim}\text{ }a/Comm(\mathcal C)(b, \mathbb D(i)) \longrightarrow a/Comm(\mathcal C)(b, \underset{i \in \mathcal I}{colim}\text{ }\mathbb D(i))
\end{equation}is an isomorphism. 
\begin{defn}\label{fpqcM9}
    A family of morphisms $\left\{\alpha_i^{op} : Spec(a_i) \longrightarrow Spec(a)\right\}_{i \in I}$ in $Aff_{\mathcal C}$ is an fpqc $\mathcal M$-cover for $Spec(a)$ if

    \smallskip
    (1) For each $i \in I$, the morphism $\alpha_i : a \longrightarrow a_i$ is $\mathcal M$-flat.

    \smallskip
    (2) There exists a finite subset $J \subseteq I$ such that the collection of functors
    \begin{equation}\label{equat3.4}
      \left\{  \alpha_j^{*}  = (a_j \boxtimes_a \_\_) : \mathbf{Mod}_{\mathcal M}(a) \longrightarrow  \mathbf{Mod}_{\mathcal M}(a_j)\right\}_{j\in J}
    \end{equation}is conservative, i.e., a morphism $\phi$ in $\mathbf{Mod}_{\mathcal M}(a)$ is an isomorphism if and   only if $\alpha_j^*(\phi)$ is an isomorphism in $\mathbf{Mod}_{\mathcal M}(a_j)$ for each $j\in J$.
\end{defn}
\begin{defn}\label{Zariskiopen}
 A morphism $\alpha^{op} : Spec(b) \longrightarrow Spec(a)$ in $Aff_{\mathcal C}$ is a spectral immersion relative to $\mathcal M$ if $\alpha$ is $\mathcal M$-flat and an epimorphism of finite type in $Comm(\mathcal C)$.
\end{defn}
\begin{lem}\label{stabilityspec}
    Spectral immersions relative to $\mathcal M$ are stable under pullbacks and closed under compositions in $Aff_{\mathcal C}$.
\end{lem}
\begin{proof}
    We need to show that $\mathcal M$-flat epimorphisms of finite type are stable under pushouts and compositions in $Comm(\mathcal C)$. We consider a pushout square in $Comm(\mathcal C)$ such that $\beta : a \longrightarrow b$ is of finite type. 
\begin{equation}\label{cocart7yz} 
\begin{tikzcd}
b'     = b \otimes_a a'                & b \arrow[l, "\alpha'"']                   \\
a' \arrow[u, "\beta'"] & a \arrow[l, "\alpha"] \arrow[u, "\beta"']
\end{tikzcd}
\end{equation} 
    We claim that $\beta' : a' \longrightarrow b' =b \otimes_a c$ is of finite type. Let $\mathbb D : \mathcal I \longrightarrow a'/Comm(\mathcal C)$ be a filtered diagram. We consider the functor $r_{\alpha} : a'/Comm(\mathcal C) \longrightarrow a/Comm(\mathcal C)$ that takes $ (a' \longrightarrow c) \mapsto (a \xrightarrow{\text{ }\alpha\text{ }} a' \longrightarrow c)$. We note that the forgetful functors $\pi_a : a/Comm(\mathcal C) \longrightarrow Comm(\mathcal C)$ and $\pi_{a'} : a'/Comm(\mathcal C) \longrightarrow Comm(\mathcal C)$ create filtered colimits and are conservative. Since $\pi_a \circ r_{\alpha} = \pi_{a'}$, it follows that $r_{\alpha}$ preserves filtered colimits. Further, $r_{\alpha}$ has a left adjoint
    \begin{equation}
      \begin{split}
       p_{\alpha} : a/Comm(\mathcal C) &\longrightarrow a'/Comm(\mathcal C)\\
                 (a \longrightarrow c) &\mapsto (a' \longrightarrow a' \otimes_{a} c) 
       \end{split}
    \end{equation}  Since $\beta$ is of finite type and $r_\alpha$ preserves filtered colimits, we have
    \begin{equation}
        \begin{split}
            \underset{i \in \mathcal I}{colim}\text{ }a'/Comm(\mathcal C)(a' \otimes_a b, \mathbb D(i)) &\cong \underset{i \in \mathcal I}{colim}\text{ }a/Comm(\mathcal C)(b, r_{\alpha}(\mathbb D(i)))\\
            &\cong a/Comm(\mathcal C)(b, \underset{i \in \mathcal I}{colim}\text{ }r_{\alpha}(\mathbb D(i)))\\
            &\cong a/Comm(\mathcal C)(b, r_{\alpha}(\underset{i \in \mathcal I}{colim}\text{ }\mathbb D(i))) \cong a'/Comm(\mathcal C)(a' \otimes_a b, \underset{i \in \mathcal I}{colim}\text{ }\mathbb D(i))
        \end{split}
    \end{equation}This shows that $\beta' : a' \longrightarrow a' \otimes_a b$ is of finite type. It may be verified that a composition of morphisms of finite type is a morphism of finite type. Using Lemma \ref{flat}, it follows that $\mathcal M$-flat epimorphisms of finite type are stable under pushouts and closed under compositions in $Comm(\mathcal C)$.
\end{proof}

\begin{defn}
    A family of morphisms $\left\{\alpha_i^{op} : Spec(a_i) \longrightarrow Spec(a)\right\}_{i \in I}$ in $Aff_{\mathcal C}$ is a spectral $\mathcal M$-cover for $Spec(a)$ if it is an fpqc $\mathcal M$-cover and for each $i \in I$, $\alpha_i^{op}$ is a spectral immersion relative to $\mathcal M$.
\end{defn}

\begin{thm}\label{checkfortopology}
    (1) The assignment
$Aff_{\mathcal C} \ni Spec(a) \mapsto \left\{\text{fpqc }\mathcal M\text{-covers for }Spec(a)\right\}$ is a basis for a Grothendieck topology on $Aff_{\mathcal C}$.

\smallskip
(2) The assignment $Aff_{\mathcal C} \ni Spec(a) \mapsto \left\{\text{spectral }\mathcal M\text{-covers for }Spec(a)\right\}$ is a basis for a Grothendieck topology on $Aff_{\mathcal C}$.
\end{thm}
\begin{proof}
    (1) If $\alpha^{op} : Spec(b) \longrightarrow Spec(a)$ is an isomorphism, it is clear that the singleton $\{\alpha^{op}\}$ is an fpqc $\mathcal M$-cover. Let $\{\alpha_i^{op} : Spec(a_i) \longrightarrow Spec(a)\}_{i \in I}$ be an fpqc $\mathcal M$-cover and $\beta^{op} : Spec(b) \longrightarrow Spec(a)$ be a morphism in $Aff_{\mathcal C}$. We consider the pushout square
\begin{equation}
    \begin{tikzcd}
a_i \otimes_a b        & b \arrow[l, "\alpha_i'"']                   \\
a_i \arrow[u, "\beta_i'"] & a \arrow[l, "\alpha_i"] \arrow[u, "\beta"']
\end{tikzcd}
\end{equation}in $Comm(\mathcal C)$ for each $i\in I$. We claim that the family of pullbacks $\{(\alpha_i')^{op} : Spec(a_i) \times_{Spec(a)} Spec(b) \cong Spec(a_i \otimes_a b) \longrightarrow Spec(b)\}_{i \in I}$ is an fpqc $\mathcal M$-cover for $Spec(b)$. Using Lemma $\ref{flat}$, it follows that each $\alpha_i' : b \longrightarrow a_i \otimes_a b$ is $\mathcal M$-flat. Further, there exists a finite subset $J \subseteq I$ such that the following collection of functors  is conservative. 
  \begin{equation}\label{equat3.4z}
      \left\{  \alpha_j^{*}  = (a_j \boxtimes_a \_\_) : \mathbf{Mod}_{\mathcal M}(a) \longrightarrow  \mathbf{Mod}_{\mathcal M}(a_j)\right\}_{j\in J}
    \end{equation}
We claim that the functor 
        $(\alpha_j'^*)_{j \in J} = ((a_j \otimes_a b) \boxtimes_b \_\_)_{j \in J} : \mathbf{Mod}_{\mathcal M}(b) \longrightarrow \prod_{j \in J} \mathbf{Mod}_{\mathcal M}(a_j \otimes_a b)$ is conservative. It suffices to show that the composite functor $({\beta_j}'_* \circ {\alpha_j'}^*)_{j \in J}$
    \begin{equation}\label{equat3.7}
        \mathbf{Mod}_{\mathcal M}(b) \xrightarrow{(\alpha_j'^*)_{j \in J}} \prod_{j \in J} \mathbf{Mod}_{\mathcal M}(a_j \otimes_a b) \xrightarrow{\prod_{j \in J} {\beta'_j}_*} \prod_{j \in J} \mathbf{Mod}_{\mathcal M}(a_j)
    \end{equation}is conservative. Using Proposition $\ref{pseudo}$, we have $({\beta_j'}_* \circ \alpha_j'^*)_{j\in J} \cong (\alpha_j^* \circ \beta_*)_{j\in J}= (\alpha_j^*)_{j \in J} \circ \beta_*$. Since $\beta_*$ is conservative, it now follows that the functor in $\eqref{equat3.7}$ is conservative.

    \smallskip
    Using Proposition \ref{pseudo} and Lemma \ref{flat}, it may be verified that composition of fpqc $\mathcal M$-covers must be an fpqc $\mathcal M$-cover. This completes the proof.

    \smallskip
    (2) The result follows from part (1) and Lemma \ref{stabilityspec}.
\end{proof}
We will refer to the   topology on $Aff_{\mathcal C}$ coming from  part (1) of Propostion $\ref{checkfortopology}$ as the fpqc $\mathcal M$-topology and denote it by $fpqc_{\mathcal M}$. The associated topology on $Aff_{\mathcal C}$ in part (2) of Propostion $\ref{checkfortopology}$ will be called the spectral $\mathcal M$-topology and  denoted by $spc_{\mathcal M}$.

\smallskip
We denote the category of presheaves of sets on $Aff_{\mathcal C}$ by $PSh(Aff_{\mathcal C})$ and let $Sh(Aff_{\mathcal C}, fpqc_{\mathcal M})$ be the full subcategory of $PSh(Aff_{\mathcal C})$ whose objects are the sheaves with respect to the   fpqc $\mathcal M$-topology. The category of sheaves on $Aff_{\mathcal C}$ with respect to the spectral 
$\mathcal M$-topology will be denoted by  $Sh(Aff_{\mathcal C})_{\mathcal M}$. We note that we have full subcategories
\begin{equation}
    Sh(Aff_{\mathcal C}, fpqc_{\mathcal M}) \subseteq Sh(Aff_{\mathcal C})_{\mathcal M} \subseteq PSh(Aff_{\mathcal C})
\end{equation} We will often use the term 
``affine scheme'' to mean either an object $Spec(a) \in Aff_{\mathcal C}$ or the corresponding presheaf on $Aff_{\mathcal C}$. Similarly, we will use 
$\alpha^{op}:Spec(a)\longrightarrow Spec(a')$ to denote both a morphism in $Aff_{\mathcal C}$ and the corresponding morphism in $PSh(Aff_{\mathcal C})$.

\medskip
We recall (see \cite[Corollary 2.11]{TV}) that for the case $(\mathcal M, \boxtimes) = (\mathcal C, \otimes)$, the spectral $\mathcal M$-topology on $Aff_{\mathcal C}$ is subcanonical, i.e. every representable presheaf $Aff_{\mathcal C}(\_\_, Spec(a))$ is a sheaf. We now give an example to show that for a general $\mathcal M$, the spectral $\mathcal M$-topology on $Aff_{\mathcal C}$ need not be subcanonical.
\begin{examples}\label{exampleactegorytop}
\emph{We consider the closed symmetric monoidal category $(\mathbf{Ab}, \otimes_{\mathbb Z}, \mathbb Z)$. Let $p, q$ be distinct primes in $\mathbb Z$ and let $\mathbb Z \xrightarrow{\text{ }\alpha\text{ }}\mathbb Z[p^{-1}]$ (resp. $\mathbb Z \xrightarrow{\text{ }\beta\text{ }}\mathbb Z[q^{-1}]$) denote the localization of $\mathbb Z$ at $\{1, p, p^2, ...\}$ (resp. $\{1, q, q^2, ...\}$). It is standard that the 
family
\begin{equation} 
       \left\{Spec(\mathbb Z[p^{-1}]) \xrightarrow{\text{ }\alpha^{op}\text{ }} Spec(\mathbb Z),\text{ }Spec(\mathbb Z[q^{-1}]) \xrightarrow{\text{ }\beta^{op}\text{ }} Spec(\mathbb Z)\right\}
\end{equation} is a cover for $Spec(\mathbb Z)$ in the Zariski topology on $Aff_{\mathbf{Ab}}$ (in the sense of \cite{TV}).}

\smallskip
\emph{We now consider the category $mod_{\mathbb Z/p\mathbb Z}$ of $(\mathbb Z/p\mathbb Z)$-modules. It is both complete and cocomplete. The bifunctor $\boxtimes : \mathbf{Ab} \times mod_{\mathbb Z/p\mathbb Z} \longrightarrow mod_{\mathbb Z/p\mathbb Z}$ given by the composition}
    \begin{equation}\label{eqtyiop3.35}
    \begin{split}
        \mathbf{Ab} \times mod_{\mathbb Z/p\mathbb Z} \xrightarrow{\text{ }\left(\_\_ \otimes_{\mathbb Z} (\mathbb Z/p\mathbb Z)\right) \times id\text{ }} mod_{\mathbb Z/p\mathbb Z} \times& mod_{\mathbb Z/p\mathbb Z} \xrightarrow{\text{ }\otimes_{\mathbb Z/p\mathbb Z}\text{ }} mod_{\mathbb Z/p\mathbb Z}\\
        (G, M) \quad\qquad\longmapsto \qquad(G \otimes_{\mathbb Z}\mathbb (\mathbb Z/p\mathbb Z),& M) \quad\quad\qquad\longmapsto \quad(G \otimes_{\mathbb Z}(\mathbb Z/p\mathbb Z)) \otimes_{\mathbb Z/p\mathbb Z} M \cong G \otimes_{\mathbb Z} M
    \end{split}
    \end{equation}\emph{determines a left $(\mathbf{Ab}, \otimes_{\mathbb Z}, \mathbb Z)$-actegory structure on $mod_{\mathbb Z/p\mathbb Z}$. Since the extension of scalars $\_\_ \otimes_{\mathbb Z} (\mathbb Z/p\mathbb Z) : \mathbf{Ab} \longrightarrow mod_{\mathbb Z/p\mathbb Z}$ along the unique ring map $\mathbb Z \longrightarrow \mathbb Z/p\mathbb Z$ is a left adjoint and the monoidal category ($mod_{\mathbb Z/p\mathbb Z}, \otimes_{\mathbb Z/p\mathbb Z}, \mathbb Z/p\mathbb Z$) is closed, it follows that the bifunctor \eqref{eqtyiop3.35} preserves colimits in both variables. Hence, we may consider the spectral $mod_{\mathbb Z/p\mathbb Z}$-topology on $Aff_{\mathbf{Ab}}$.}

\smallskip

\smallskip
\emph{We consider again, the localizations}
\begin{equation}\label{localizationsatprimepowers}\mathbb Z
\xrightarrow{\text{ }\alpha\text{ }} \mathbb Z[p^{-1}],\quad \mathbb Z
\xrightarrow{\text{ }\beta\text{ }} \mathbb Z[q^{-1}]
\end{equation}
\emph{We make the following observations :}

\emph{(i) The category $\mathbf{Mod}_{mod_{\mathbb Z/p\mathbb Z}}(\mathbb Z[p^{-1}])$ is the terminal category (with exactly one object and one arrow). To see this, we note that an object of $\mathbf{Mod}_{mod_{\mathbb Z/p\mathbb Z}}(\mathbb Z[p^{-1}])$ is a $(\mathbb Z/p\mathbb Z)$-module $M$ together with a $(\mathbb Z/p\mathbb Z)$-linear map}
\begin{equation*}
    \mathbb Z[p^{-1}] \boxtimes M \cong \mathbb Z[p^{-1}] \otimes_{\mathbb Z} M \xrightarrow{\text{ }\rho\text{ }} M
\end{equation*}
\emph{such that \eqref{newnameoldstuff} holds. In particular, the map $\rho$ is a split epimorphism in $mod_{\mathbb Z/p\mathbb Z}$ and hence a surjection. Further, since $M$ is a $(\mathbb Z/p\mathbb Z)$-module,  $\mathbb Z[p^{-1}] \otimes_{\mathbb Z} M$ is the zero module. It follows that $M$ is the zero module.} \emph{Hence, $\mathbf{Mod}_{mod_{\mathbb Z/p\mathbb Z}}(\mathbb Z[p^{-1}])$ is the terminal category.} 

\smallskip
\emph{(ii) Since $\mathbb Z$ is the unit object of $(\mathbf{Ab}, \otimes_{\mathbb Z}, \mathbb Z)$, the restriction of scalars functors along \eqref{localizationsatprimepowers}
\begin{equation*}
\begin{split}
\alpha_* : \mathbf{Mod}_{mod_{\mathbb Z/p\mathbb Z}}(\mathbb Z[p^{-1}]) &\longrightarrow \mathbf{Mod}_{mod_{\mathbb Z/p\mathbb Z}}(\mathbb Z) \cong mod_{\mathbb Z/p\mathbb Z}\\
\beta_* : \mathbf{Mod}_{mod_{\mathbb Z/p\mathbb Z}}(\mathbb Z[q^{-1}]) &\longrightarrow \mathbf{Mod}_{mod_{\mathbb Z/p\mathbb Z}}(\mathbb Z) \cong mod_{\mathbb Z/p\mathbb Z}
\end{split}
\end{equation*}are naturally isomorphic to forgetful functors, as in Remark \ref{creation}. It follows from the uniqueness of left adjoints that the corresponding extension of scalars functors are given by}
\begin{equation*}
\begin{split}
    mod_{\mathbb Z/p\mathbb Z} \cong \mathbf{Mod}_{mod_{\mathbb Z/p\mathbb Z}}(\mathbb Z) &\xrightarrow{\text{ }\alpha^*\text{ }} \mathbf{Mod}_{mod_{\mathbb Z/p\mathbb Z}}(\mathbb Z[p^{-1}]),\text{ }M \mapsto \mathbb Z[p^{-1}] \boxtimes_{\mathbb Z} M \cong \mathbb Z[p^{-1}] \otimes_{\mathbb Z} M\\
    mod_{\mathbb Z/p\mathbb Z} \cong \mathbf{Mod}_{mod_{\mathbb Z/p\mathbb Z}}(\mathbb Z) &\xrightarrow{\text{ }\beta^*\text{ }} \mathbf{Mod}_{mod_{\mathbb Z/p\mathbb Z}}(\mathbb Z[q^{-1}]),\text{ }M \mapsto \mathbb Z[q^{-1}] \boxtimes_{\mathbb Z} M \cong \mathbb Z[q^{-1}] \otimes_{\mathbb Z} M
\end{split}
\end{equation*}\emph{Using observation (ii), we note that the following diagram}\begin{equation}\label{fhy3.42}
\begin{tikzcd}
	{\mathbf{Mod}_{mod_{\mathbb Z/p\mathbb Z}}(\mathbb Z)} && {\mathbf{Mod}_{mod_{\mathbb Z/p\mathbb Z}}(\mathbb Z[p^{-1}]) \times \mathbf{Mod}_{mod_{\mathbb Z/p\mathbb Z}}(\mathbb Z[q^{-1}])} \\
	{mod_{\mathbb Z/p\mathbb Z}} && {mod_{\mathbb Z/p\mathbb Z} \times mod_{\mathbb Z/p\mathbb Z}} \\
	{\mathbf{Ab}} && {\mathbf{Ab} \times \mathbf{Ab}}
	\arrow["{(\alpha^*, \beta^*)}", from=1-1, to=1-3]
	\arrow["\sim"', from=1-1, to=2-1]
	\arrow["{\emph{forgetful}\text{ }\times\text{ }\emph{forgetful}}", from=1-3, to=2-3]
	\arrow["{\emph{forgetful}}"', from=2-1, to=3-1]
	\arrow["{\emph{forgetful}\text{ }\times\text{ }\emph{forgetful}}", from=2-3, to=3-3]
	\arrow["{(\mathbb Z[p^{-1}] \otimes_{\mathbb Z} \_\_,\text{ }\mathbb Z[q^{-1}] \otimes_{\mathbb Z} \_\_)}"', from=3-1, to=3-3]
\end{tikzcd}
\end{equation}\emph{commutes upto a natural isomorphism. It is clear that all the functors in diagram \eqref{fhy3.42} different from}
\begin{equation}\label{pairfunctorusedhere3.32}
    (\alpha^*, \beta^*) : \mathbf{Mod}_{mod_{\mathbb Z/p\mathbb Z}}(\mathbb Z) \longrightarrow \mathbf{Mod}_{mod_{\mathbb Z/p\mathbb Z}}(\mathbb Z[p^{-1}]) \times \mathbf{Mod}_{mod_{\mathbb Z/p\mathbb Z}}(\mathbb Z[q^{-1}])
\end{equation}\emph{are conservative. It follows that the functor \eqref{pairfunctorusedhere3.32} is also conservative. Since $\mathbf{Mod}_{mod_{\mathbb Z/p\mathbb Z}}(\mathbb Z[p^{-1}])$ is the terminal category, it follows that $\beta^*$ is conservative.}

\smallskip
\emph{Moreover, diagram \eqref{fhy3.42} reduces to the following commutative diagram}
\begin{equation}\label{userandomdiagrum3.33}
\begin{tikzcd}
\mathbf{Mod}_{mod_{\mathbb Z/p\mathbb Z}}(\mathbb Z) \arrow[d, "\beta^*"'] \arrow[rr, "\sim"] &  & mod_{\mathbb Z/p\mathbb Z} \arrow[rr, "\emph{forgetful}"] &  & \mathbf{Ab} \arrow[d, "{\mathbb Z[q^{-1}] \otimes_{\mathbb Z} \_\_}"] \\
\mathbf{Mod}_{mod_{\mathbb Z/p\mathbb Z}}(\mathbb Z[q^{-1}]) \arrow[rr, "\emph{forgetful}"]                                                          &  & mod_{\mathbb Z/p\mathbb Z} \arrow[rr, "\emph{forgetful}"] &  & \mathbf{Ab}                                                    
\end{tikzcd}
\end{equation}\emph{Since $\mathbb Z[q^{-1}]$ is a flat $\mathbb Z$-module (in the ordinary sense), the functor $\mathbb Z[q^{-1}] \otimes_{\mathbb Z} \_\_ : \mathbf{Ab} \longrightarrow \mathbf{Ab}$ preserves finite limits. Further, since the forgetful functors in \eqref{userandomdiagrum3.33} are conservative and preserve finite limits, it follows that
$\beta^* : \mathbf{Mod}_{mod_{\mathbb Z/p\mathbb Z}}(\mathbb Z) \longrightarrow \mathbf{Mod}_{mod_{\mathbb Z/p\mathbb Z}}(\mathbb Z[q^{-1}])$ preserves finite limits. Hence,} $\beta : \mathbb Z \longrightarrow \mathbb Z[q^{-1}]$ \emph{is $mod_{\mathbb Z/p\mathbb Z}$-flat.}

\smallskip
\emph{It follows that the singleton family}
\begin{equation}\label{weirdenuf}
    \left\{Spec(\mathbb Z[q^{-1}]) \xrightarrow{\text{ }\beta^{op}\text{ }} Spec(\mathbb Z)\right\}
\end{equation}\emph{is a spectral $mod_{\mathbb Z/p\mathbb Z}$-cover for $Spec(\mathbb Z)$}. \emph{Our next result shows that the spectral $mod_{\mathbb Z/p\mathbb Z}$-topology on $Aff_{\mathbf{Ab}}$ is not subcanonical.}
\end{examples}
\begin{thm} 
There exists a commutative ring $R$ such that the representable presheaf $Aff_{\mathbf{Ab}}(\_\_, Spec(R))$ is not a sheaf for the spectral $mod_{\mathbb Z/p\mathbb Z}$-topology on $Aff_{\mathbf{Ab}}$. In other words, the spectral $mod_{\mathbb Z/p\mathbb Z}$-topology on $Aff_{\mathbf{Ab}}$ is not subcanonical.
\end{thm}
\begin{proof} 
We fix a commutative ring $R$ such that $Comm(\mathbf{Ab})(R, \mathbb Z) \ncong Comm(\mathbf{Ab})(R, \mathbb Z[q^{-1}])$ as sets. Suppose that $Aff_{\mathbf{Ab}}(\_\_, Spec(R)) : Aff_{\mathbf{Ab}}^{op} \longrightarrow Set$ is a sheaf for the spectral $mod_{\mathbb Z/p\mathbb Z}$-topology on $Aff_{\mathbf{Ab}}$. We consider the spectral $mod_{\mathbb Z/p\mathbb Z}$-cover
\begin{equation}
    \left\{Spec(\mathbb Z[q^{-1}]) \xrightarrow{\text{ }\beta^{op}\text{ }} Spec(\mathbb Z)\right\}
\end{equation}for $Spec(\mathbb Z)$. We have an equalizer in $Set$
    \begin{equation}
\begin{tikzcd}
{Aff_{\mathbf{Ab}}(Spec(\mathbb Z), Spec(R))} \arrow[dd, "\_\_ \circ \beta^{op}"', dashed] \arrow[rrr, no head, equal]                                                                    &  &  & {Comm(\mathbf{Ab})(R, \mathbb Z)} \arrow[dd, "\beta \circ \_\_", dashed]                                                                       \\
                                                                                                                                                                                                      &  &  &                                                                                                                                              \\
{Aff_{\mathbf{Ab}}(Spec(\mathbb Z[q^{-1}]), Spec(R))} \arrow[dd, "\_\_ \circ \iota_1^{op}"', shift right=3] \arrow[dd, "\_\_ \circ \iota_2^{op}", shift left=3] \arrow[rrr, no head, equal] &  &  & {Comm(\mathbf{Ab})(R, \mathbb Z[q^{-1}])} \arrow[dd, "\iota_1 \circ \_\_"', shift right=3] \arrow[dd, "\iota_2 \circ \_\_", shift left=3] \\
                                                                                                                                                                                                      &  &  &                                                                                                                                              \\
{Aff_{\mathbf{Ab}}(Spec(\mathbb Z[q^{-1}]) \times_{Spec(\mathbb Z)} Spec(\mathbb Z[q^{-1}]), Spec(R))} \arrow[rrr, no head, equal]                                                                    &  &  & {Comm(\mathbf{Ab})(R, \mathbb Z[q^{-1}] \otimes_{\mathbb Z} \mathbb Z[q^{-1}])}                                                                          
\end{tikzcd}
    \end{equation}where 
\begin{tikzcd}
\mathbb Z[q^{-1}] \arrow[rr, "\iota_1", shift left=2] \arrow[rr, "\iota_2"', shift right=2] &  & \mathbb Z[q^{-1}] \otimes_{\mathbb Z} \mathbb Z[q^{-1}]
\end{tikzcd} are the canonical maps associated to the pushout $\mathbb Z[q^{-1}] \otimes_{\mathbb Z} \mathbb Z[q^{-1}]$. Since the localization $\beta : \mathbb Z \longrightarrow \mathbb Z[q^{-1}]$ is an epimorphism in $Comm(\mathbf{Ab})$, it follows that $\iota_1 = \iota_2 : \mathbb Z[q^{-1}] \longrightarrow \mathbb Z[q^{-1}] \otimes_{\mathbb Z} \mathbb Z[q^{-1}]$ is an isomorphism. Hence,
\begin{equation}\small
    Comm(\mathbf{Ab})(R, \mathbb Z[q^{-1}]) \cong Eq\left(
\begin{tikzcd}
{Comm(\mathbf{Ab})(R, \mathbb Z[q^{-1}])} \arrow[rr, "\iota_1 \circ \_\_", shift left=2] \arrow[rr, "\iota_2 \circ \_\_"', shift right=2] &  & {Comm(\mathbf{Ab})(R, \mathbb Z[q^{-1}] \otimes_{\mathbb Z} \mathbb Z[q^{-1}])}
\end{tikzcd} \right) \cong Comm(\mathbf{Ab})(R, \mathbb Z)
\end{equation}which contradicts the choice of $R$. It follows that $Aff_{\mathbf{Ab}}(\_\_, Spec(R))$ is not a sheaf with respect to the spectral $mod_{\mathbb Z/p\mathbb Z}$-topology on $Aff_{\mathbf{Ab}}$. The proof is now complete.
\end{proof}
\section{Schemes relative to $\mathcal M$}

From now onwards, we will say that the left $\mathcal C$-actegory $(\mathcal M, \boxtimes)$ is subcanonical  if the fpqc $\mathcal M$-topology (and hence, the spectral $\mathcal M$-topology)   on $Aff_{\mathcal C}$ is subcanonical, i.e., for each $Spec(a)\in Aff_{\mathcal C}$, the representable presheaf $Aff_{\mathcal C}(\_\_, Spec(a))$ on $Aff_{\mathcal C}$ is a sheaf  for the fpqc $\mathcal M$-topology on $Aff_{\mathcal C}$. 
 In this section, we present the main definition of this paper, that of a scheme relative to $\mathcal M$. We assume throughout this section that the left $\mathcal C$-actegory $\mathcal M$ is subcanonical. We begin with the following lemma.
 
 \begin{lem}\label{coverinducepiintopos}
For every spectral $\mathcal M$-cover $\{\alpha_i^{op} : Spec(a_i) \longrightarrow Spec(a)\}_{i \in I}$, the canonical morphism 
\begin{equation}\coprod_{i \in I} Aff_{\mathcal C}(\_\_, Spec(a_i)) \xrightarrow{\quad\coprod_{i \in I} Aff_{\mathcal C}(\_\_,\alpha_i^{op})=\coprod_{i \in I}\alpha_i^{op}\quad}  Aff_{\mathcal C}(\_\_, Spec(a))
\end{equation} is an epimorphism  in $Sh(Aff_{\mathcal C})_{\mathcal M}$. 
\end{lem} 
\begin{proof}
    Let $p, q$ be morphisms in  $Sh(Aff_{\mathcal C})_{\mathcal M}$ such that the compositions
    \begin{equation}\label{eqtui4.1}
    \begin{array}{c}
        \coprod_{i \in I} Aff_{\mathcal C}(\_\_, Spec(a_i)) \xrightarrow{\quad\coprod_{i \in I}\alpha_i^{op}\quad}  Aff_{\mathcal C}(\_\_, Spec(a)) \xrightarrow{\text{ }p\text{ }} F\\
        \coprod_{i \in I} Aff_{\mathcal C}(\_\_, Spec(a_i)) \xrightarrow{\quad\coprod_{i \in I}\alpha_i^{op}\quad}   Aff_{\mathcal C}(\_\_, Spec(a)) \xrightarrow{\text{ }q\text{ }} F\\
        \end{array}
    \end{equation}are equal. Then, \eqref{eqtui4.1} implies that for each $i \in I$, we have $ p \circ \alpha_i^{op} = q \circ \alpha_i^{op}$ in $Sh(Aff_{\mathcal C})_{\mathcal M}$.
 Let $p'$ (resp. $q'$) be the element of $F(Spec(a))$ corresponding under  Yoneda lemma to   $p$ (resp. $q$). It follows that for each $i \in I$,  $p \circ \alpha_i^{op} = q \circ \alpha_i^{op}$ corresponds under  Yoneda Lemma to the element $F(\alpha_i^{op})(p') = F(\alpha_i^{op})(q') \in F(Spec(a_i))$. Since $F \in Sh(Aff_{\mathcal C})_{\mathcal M}$ and  $\{\alpha_i^{op} : Spec(a_i) \longrightarrow Spec(a)\}_{i \in I}$  is a spectral $\mathcal M$-cover, we see that the map 
 \begin{equation} F(Spec(a)) \longrightarrow \prod_{i \in I}F(Spec(a_i)) \qquad x \mapsto (F(\alpha_i^{op})(x))_{i \in I}
 \end{equation} must be  injective. It follows that $p' = q'$ and hence $p = q$.  
\end{proof}
 
\begin{defn}\label{zariskiopenforsheaves}
   (1) Let $X := Aff_{\mathcal C}(\_\_, Spec(a))$ be an affine scheme and $F \subseteq X$ a subsheaf in $Sh(Aff_{\mathcal C})_{\mathcal M}$. We say that $F\subseteq X$ is a Zariski open  relative to $\mathcal M$   if there is a family $\{Spec(a_i) \longrightarrow Spec(a)\}_{i \in I}$ of spectral immersions relative to $\mathcal M$  such that the induced morphism $\coprod_{i \in I} Aff_{\mathcal C}(\_\_, Spec(a_i)) \longrightarrow Aff_{\mathcal C}(\_\_, Spec(a))=X$ has image $F$, i.e., we have
    \begin{equation}
        \coprod_{i \in I} Aff_{\mathcal C}(\_\_, Spec(a_i)) \xrightarrow{\mbox{$\quad$ epi $\quad$ }} F \hookrightarrow Aff_{\mathcal C}(\_\_, Spec(a))=X
    \end{equation} in $Sh(Aff_{\mathcal C})_{\mathcal M}$.

    \smallskip
    (2) A morphism $F \longrightarrow G$ in $Sh(Aff_{\mathcal C})_{\mathcal M}$ is a Zariski open immersion relative to $\mathcal M$ if for each $X\in Aff_{\mathcal C}$ and each morphism $X \longrightarrow G$, the induced morphism $F \times_G X \longrightarrow X$ in $Sh(Aff_{\mathcal C})_{\mathcal M}$ is a monomorphism whose image is Zariski open relative to $\mathcal M$ in $X$.
    \end{defn}
    \begin{lem}\label{grotop}
(1) Let $G \longrightarrow H$ be a morphism in $Sh(Aff_{\mathcal C})_{\mathcal M}$. Then, the functor $\_\_ \times_H G : Sh(Aff_{\mathcal C})_{\mathcal M}/H \longrightarrow Sh(Aff_{\mathcal C})_{\mathcal M}/G$ sending $(F \longrightarrow H) \mapsto (F \times_H G \longrightarrow G)$ preserves colimits.

\smallskip
(2) Let  $\coprod_{i \in I} F_i \longrightarrow H$ be an epimorphism in $Sh(Aff_{\mathcal C})_{\mathcal M}$. If $G\longrightarrow H$ is a morphism in  $Sh(Aff_{\mathcal C})_{\mathcal M}$, then the induced morphism $\coprod_{i \in I} F_i \times_H G \longrightarrow G$ is an epimorphism  in  $Sh(Aff_{\mathcal C})_{\mathcal M}$.
\end{lem}
\begin{proof}
(1) It is easy to verify that the analogous result holds in the category of sets. Since colimits in $PSh(Aff_{\mathcal C})$ are computed objectwise, we see that the functor
$\_\_ \times_H G : PSh(Aff_{\mathcal C})/H \longrightarrow PSh(Aff_{\mathcal C})/G$ sending $(F \longrightarrow H) \mapsto (F \times_H G \longrightarrow G)$ preserves colimits.  Let  $(.)^{++}_{\mathcal M}: PSh(Aff_{\mathcal C}) \longrightarrow Sh(Aff_{\mathcal C})_{\mathcal M}$  be the sheafification functor which is left adjoint 
to the inclusion $Sh(Aff_{\mathcal C})_{\mathcal M} \hookrightarrow PSh(Aff_{\mathcal C})$. In particular, the functor $(.)^{++}_{\mathcal M}$ preserves colimits. Since sheafification is obtained from  filtered
 colimits (see for instance, \cite[$\S$ 3.5]{MM}), we know that it also preserves finite limits. The result is now clear.  

    \smallskip
(2) It follows from part (1) that $\coprod_i F_i \times_H G \cong \coprod_i \left(F_i \times_H G\right)$ in $Sh(Aff_{\mathcal C})_{\mathcal M}/G$. Since   $\coprod_i F_i \longrightarrow H$ is an epimorphism in $Sh(Aff_{\mathcal C})_{\mathcal M}$, it induces an epimorphism $(\coprod_i F_i \longrightarrow H) \longrightarrow (H\xrightarrow{id} H)$ in $Sh(Aff_{\mathcal C})_{\mathcal M}/H$. Using part (1), it follows that the morphism $(\coprod_i F_i \times_H G \cong \coprod_i (F_i \times_H G)\longrightarrow G) \longrightarrow (G\xrightarrow{id}G)$ is an epimorphism in $Sh(Aff_{\mathcal C})_{\mathcal M}/G$ and hence  $\coprod_{i \in I} F_i \times_H G \longrightarrow G$ is an epimorphism in $Sh(Aff_{\mathcal C})_{\mathcal M}$. This completes the proof.
\end{proof}
\begin{lem}\label{zaropenimmono}
 A Zariski open immersion $f : F \longrightarrow G$ relative to $\mathcal M$  is a monomorphism in $Sh(Aff_{\mathcal C})_{\mathcal M}$.
\end{lem}
\begin{proof}
  Let $p : H \longrightarrow F, q : H \longrightarrow F$ be morphisms in $Sh(Aff_{\mathcal C})_{\mathcal M}$ such that $f \circ p = f \circ q$. Suppose that $p \neq q$. Then, there exists   $Spec(a) \in Aff_{\mathcal C}$ with $p_{Spec(a)} \neq q_{Spec(a)}$. Hence, there must be an element $x \in H(Spec(a))$ with $p_{Spec(a)}(x) \neq q_{Spec(a)}(x)$.
 We set $Y = Aff_{\mathcal C}(\_\_, Spec(a))$. Let $h : Y \longrightarrow H$ (resp. $g : Y \longrightarrow G$) be the morphism corresponding under Yoneda lemma to the element $x \in H(Spec(a))$ (resp. $f_{Spec(a)}(p_{Spec(a)}(x)) \in G(Spec(a))$). It follows that 
       \begin{equation}\label{eqtuni4.5}
       f \circ p \circ h = g
       \end{equation}Further, since $p_{Spec(a)}(x) \neq q_{Spec(a)}(x)$, we have
       \begin{equation}\label{eqtuni4.6}
       p \circ h \neq q \circ h
       \end{equation}We  consider the pullback squares in $Sh(Aff_{\mathcal C})_{\mathcal M}$
\begin{equation}
\begin{tikzcd}
F \times_G Y \arrow[rr, "f'"] \arrow[d, "g'"] &  & Y \arrow[d, "g"] &  & H \times_G Y \arrow[rr, "t"] \arrow[d, "g''"] &  & Y \arrow[d, "g"] \\
F \arrow[rr, "f"]                             &  & G                &  & H \arrow[rr, "f \circ p = f \circ q"]         &  & G               
\end{tikzcd}
\end{equation}
Using the universal property of $F \times_G Y$, the pair $(p \circ g'' : H \times_G Y \longrightarrow F,\text{ }t : H \times_G Y \longrightarrow Y)$ (resp. $(q \circ g'' : H \times_G Y \longrightarrow F,\text{ }t : H \times_G Y \longrightarrow Y)$) induces a unique morphism
\begin{equation}
    p' : H \times_G Y \longrightarrow F \times_G Y,\quad(\text{ resp. } q' : H \times_G Y \longrightarrow F \times_G Y)
\end{equation}
such that $p \circ g'' = g' \circ p'$ and $ f' \circ p' = t$ (resp. $q \circ g'' = g' \circ q'$ and $f' \circ q' = t$).

\smallskip
By Definition \ref{zariskiopenforsheaves}(2), since $f$ is a Zariski open immersion relative to $\mathcal M$, we see that $f'$ is a monomorphism in $Sh(Aff_{\mathcal C})_{\mathcal M}$.  Hence, the equality $f' \circ p' = t = f' \circ q'$ implies that $p' = q'$. It follows that
\begin{equation}\label{eqtuni4.9} 
p \circ g'' = g' \circ p' = g' \circ q' = q \circ g''
\end{equation}

\smallskip
Using \eqref{eqtuni4.5} and the universal property of   $H \times_G Y$, the pair $(h : Y \longrightarrow H, 1_Y : Y \longrightarrow Y)$ induces a unique morphism $u : Y \longrightarrow H \times_G Y$ such that $t \circ u = 1_Y$ and $g'' \circ u = h$. It follows from \eqref{eqtuni4.9} that
\begin{equation} 
p \circ h = p \circ g'' \circ u = q \circ g'' \circ u = q \circ h
\end{equation}which contradicts \eqref{eqtuni4.6}. Hence $p = q$. This proves the result.
\end{proof}
\begin{lem}\label{lema4.2}
In $Sh(Aff_{\mathcal C})_{\mathcal M}$,  Zariski open immersions relative to $\mathcal M$ are stable under base change and closed under composition.
\end{lem}
\begin{proof} 
It is easy to see that Zariski open immersions relative to $\mathcal M$ are stable under base change.

\smallskip
Suppose that $f:F \longrightarrow G$ and $g:G \longrightarrow H$ are Zariski open immersions relative to $\mathcal M$. We need to show that $g\circ f$ is a Zariski open immersion relative to $\mathcal M$. Let $h:X \longrightarrow H$ be a morphism in $Sh(Aff_{\mathcal C})_{\mathcal M}$ where $X\in Aff_{\mathcal C}$. We consider the following pullback diagram
\begin{equation}
\xymatrix{
F \times_H X \cong F \times_G (G \times_H X) \ar[rr]^{\qquad \qquad f'} \ar[d]^{h''} &  & G \times_H X \ar[rr]^{\quad g'} \ar[d]^{h'} &  & X \ar[d]^h \\
F \ar[rr]^f                                             &  & G \ar[rr]^g                      &  & H  \\
}
\end{equation}It follows from Lemma \ref{zaropenimmono} that both $f$ and $g$ are monomorphisms. Hence,  $g'\circ f':F \times_H X \cong F \times_G (G \times_H X) \longrightarrow G \times_H X \longrightarrow X$ is a monomorphism in $Sh(Aff_{\mathcal C})_{\mathcal M}$. It remains to show that $F \times_H X \subseteq X$ is Zariski open relative to $\mathcal M$.

\smallskip
Since  $g:G \longrightarrow H$ is a Zariski open immersion relative to $\mathcal M$, we see that $g':G \times_H X \longrightarrow X$ is a monomorphism whose image is Zariski open relative to $\mathcal M$. Hence, there exists a family $\{\alpha_i: X_i \longrightarrow X\}_{i \in I}$ of spectral immersions relative to $\mathcal M$ such that  $\coprod_i \alpha_i: \coprod_i X_i \longrightarrow X$ has image $G \times_H X$ in $Sh(Aff_{\mathcal C})_{\mathcal M}$. For each $i'\in I$, we consider the composition $X_{i'} \longrightarrow \coprod_{i\in I} X_{i} \twoheadrightarrow  G \times_H X$ and form the following pullback diagram in $Sh(Aff_{\mathcal C})_{\mathcal M}$
\begin{equation}\label{412edw}
\xymatrix{
F \times_G X_{i'} \cong (F \times_H X) \times_{G \times_H X} X_{i'} \ar[rrr] \ar[d] &  &  & F \times_H X \ar[d]^{f'} \ar[rr]^{h''} &  & F \ar[d]^{f} \\
X_{i'} \ar[rrr]                                                                     &  &  & G \times_H X \ar[rr]^{h'}          &  & G          
}
\end{equation}
Since Zariski open immersions relative to $\mathcal M$ are stable under pullbacks, the morphism $f':F \times_H X \longrightarrow G \times_H X$ is a Zariski open immersion relative to $\mathcal M$. It follows that the morphism
\begin{equation}
    F \times_G X_{i'} \cong (F \times_G (G \times_H X)) \times_{G \times_H X} X_{i'} \longrightarrow X_{i'}
\end{equation} in \eqref{412edw} is a monomorphism whose image is Zariski open relative to $\mathcal M$. Hence, there is a family $\{Y_{i'j} \longrightarrow X_{i'}\}_{j \in J_{i'}}$ of spectral immersions relative to $\mathcal M$ such that the image of the canonical morphism $\coprod_{j \in J_{i'}} Y_{i'j} \longrightarrow X_{i'}$  in $Sh(Aff_{\mathcal C})_{\mathcal M}$  is $F \times_G X_{i'}$. Using Lemma \ref{stabilityspec} and Lemma \ref{grotop}, it may be verified that
\begin{equation} 
\{Y_{i'j} \longrightarrow X_{i'}\longrightarrow X\}_{j \in J_{i'}, i' \in I}
\end{equation}is a family of spectral immersions relative to $\mathcal M$ such that the image of the canonical morphism $\coprod_{j \in J_{i'}, i' \in I} Y_{i'j} \longrightarrow \coprod_{i' \in I} X_{i'} \longrightarrow X$ is $F \times_H X$. This shows that $F \times_H X \subseteq X$ is Zariski open relative to $\mathcal M$. This completes the proof.
\end{proof}

\begin{thm}\label{zaropenim}
Let $X\in Aff_{\mathcal C}$ be an affine scheme.

\smallskip
(1) Let $f : F \longrightarrow X$ be a Zariski open immersion relative to $\mathcal M$. Then, the image of $f$ is Zariski open relative to $\mathcal M$.

\smallskip
(2) Let $F \in Sh(Aff_{\mathcal C})_{\mathcal M}$ be a subsheaf of $X$ that is Zariski open relative to $\mathcal M$. Then, the inclusion $F \hookrightarrow X$ is a Zariski open immersion relative to $\mathcal M$.
\end{thm}
\begin{proof}
(1) The proof is straightforward.

\smallskip
(2) Since $F \subseteq X$ is Zariski open relative to $\mathcal M$, there exists a family $\{X_i \longrightarrow X\}_{i \in I}$ of spectral immersions relative to $\mathcal M$ such that the image of the canonical morphism $\coprod_i X_i \longrightarrow X$ in $Sh(Aff_{\mathcal C})_{\mathcal M}$ is $F$. Let $Z\in Aff_{\mathcal C}$ and $ Z \longrightarrow X$ be a morphism in $Sh(Aff_{\mathcal C})_{\mathcal M}$. For each $i' \in I$, we have a pullback diagram in $Sh(Aff_{\mathcal C})_{\mathcal M}$
\begin{equation}\label{415ewd}
\begin{tikzcd}
X_{i'} \times_X Z \arrow[d] \arrow[r] & F \times_X Z \arrow[r, hook] \arrow[d] & Z \arrow[d] \\
X_{i'} \arrow[r]                      & F \arrow[r, hook]                & X          
\end{tikzcd}
\end{equation}where the morphism $X_{i'}\longrightarrow F$ is the composition $X_{i'} \longrightarrow \coprod_{i \in I} X_{i} \twoheadrightarrow F$. We note that each $X_{i'} \times_X Z$ is affine. Using Lemma \ref{stabilityspec}, it follows that each composition
\begin{equation}
X_{i'} \times_X Z \longrightarrow F \times_X Z \longrightarrow Z
\end{equation} in \eqref{415ewd} is a spectral immersion relative to $\mathcal M$. Further, by Lemma $\ref{grotop}$, the induced morphism $\coprod_{i'\in I} X_{i'} \times_X Z \longrightarrow Z$ has image $F \times_X Z$. Hence, $F \times_X Z \subseteq Z$ is Zariski open relative to $\mathcal M$. This shows that the inclusion $F \hookrightarrow X$ is a Zariski open immersion relative to $\mathcal M$.
\end{proof}

Our next result shows that spectral immersions relative to $\mathcal M$ are also Zariski open immersions relative to $\mathcal M$ in $Sh(Aff_{\mathcal C})_{\mathcal M}$.

    \begin{thm}\label{oneimpliesother}
        Let $\alpha^{op} : Spec(a) \longrightarrow Spec(b)$ be a spectral immersion relative to $\mathcal M$. Then $\alpha^{op}$ considered as a morphism $Z = Aff_{\mathcal C}(\_\_, Spec(a)) \longrightarrow Aff_{\mathcal C}(\_\_, Spec(b)) = Y$ is a Zariski open immersion relative to $\mathcal M$ in $Sh(Aff_{\mathcal C})_{\mathcal M}$.
    \end{thm}
\begin{proof}
     Let $X =Spec(c)$ be an affine scheme and $g : X \longrightarrow Y$ be a morphism in $Sh(Aff_{\mathcal C})_{\mathcal M}$. Using Lemma \ref{stabilityspec}, the morphism
     \begin{equation}\label{eqtuni4.17}
    Z\times_YX =   Spec(a \otimes_b c) \cong Spec(a) \times_{Spec(b)} Spec(c) \longrightarrow Spec(c)=X
     \end{equation}is a spectral immersion relative to $\mathcal M$ in $Aff_{\mathcal C}$, and in particular a monomorphism in $Aff_{\mathcal C}$. Since the Yoneda embedding $Aff_{\mathcal C} \hookrightarrow PSh(Aff_{\mathcal C})$ preserves limits, the morphism 
     \begin{equation}
       Z \times_Y X \cong Aff_{\mathcal C}(\_\_, Spec(a) \times_{Spec(b)} Spec(c)) \cong Aff_{\mathcal C}(\_\_, Spec(a \otimes_b c)) \longrightarrow Aff_{\mathcal C}(\_\_, Spec(c)) \cong X
     \end{equation}is a monomorphism in $PSh(Aff_{\mathcal C})$, and hence in $Sh(Aff_{\mathcal C})_{\mathcal M}$. Further, since the morphism in \eqref{eqtuni4.17} is a spectral immersion relative to $\mathcal M$, the subsheaf $Z \times_Y X \subseteq X$ is Zariski open relative to $\mathcal M$ by Definition \ref{zariskiopenforsheaves}(1). Hence, $\alpha^{op}$ is a Zariski open immersion relative to $\mathcal M$.
\end{proof}

We are ready to define the notion of a scheme relative to $\mathcal M$.
\begin{defn}\label{defin4.8}
(1)  Let $F \in Sh(Aff_{\mathcal C})_{\mathcal M}$. An affine Zariski $\mathcal M$-covering of $F$ is a family $\left\{X_i \longrightarrow F\right\}_{i\in I}$ of morphisms in $Sh(Aff_{\mathcal C})_{\mathcal M}$ where

\smallskip
(i) For each $i \in I$, $X_i\in Aff_{\mathcal C}$ is an affine scheme and   $X_i \longrightarrow F$ is a Zariski open immersion relative to $\mathcal M$.

\smallskip
(ii) The induced morphism $\coprod_{i \in I} X_i \longrightarrow F$ is an epimorphism in $Sh(Aff_{\mathcal C})_{\mathcal M}$.

\vspace{0.2in}

(2)    An object $F \in Sh(Aff_{\mathcal C})_{\mathcal M}$ is an $\mathcal M$-scheme if it admits an affine Zariski $\mathcal M$-covering.
\end{defn}
It is clear that every affine  $X\in Aff_{\mathcal C}$ is an $\mathcal M$-scheme. We let    $Sch(\mathcal C)_{\mathcal M}$ be the full subcategory of $Sh(Aff_{\mathcal C})_{\mathcal M}$ whose objects are   $\mathcal M$-schemes. We now prove some properties of $Sch(\mathcal C)_{\mathcal M}$.
\begin{lem}\label{lema4.5}
(1) Let $X, Y \in Aff_{\mathcal C}$ be affine schemes and $Y \longrightarrow X \longleftarrow G$ be morphisms in $Sh(Aff_{\mathcal C})_{\mathcal M}$. If $G$ is an $\mathcal M$-scheme, then $F := Y \times_X G$ in $Sh(Aff_{\mathcal C})_{\mathcal M}$ is an $\mathcal M$-scheme.

\smallskip
(2) Suppose that $F \longrightarrow F_0$ is a morphism in $Sh(Aff_{\mathcal C})_{\mathcal M}$ where $F_0 \in Sch(\mathcal C)_{\mathcal M}$. If $\{g_i : X_i \longrightarrow F_0\}_{i \in I}$ is an affine Zariski $\mathcal M$-covering of $F_0$ such that each pullback $X_i \times_{F_0} F$ is an $\mathcal M$-scheme, then $F$ is an $\mathcal M$-scheme. 
\end{lem}
\begin{proof}
(1) Since $G$ is an $\mathcal M$-scheme, it has an affine Zariski $\mathcal M$-covering $\{h_i : X_i \longrightarrow G\}_ {i \in I}$. For each $i \in I$, we have a pullback diagram in $Sh(Aff_{\mathcal C})_{\mathcal M}$
\begin{equation}\label{dig4.19}
\begin{tikzcd}
X_i \times_X Y \cong X_i \times_G F \arrow[rr] \arrow[d] &  & F \arrow[d] \arrow[rr] &  & Y \arrow[d] \\
X_i \arrow[rr, "h_i"]                                           &  & G \arrow[rr]           &  & X          
\end{tikzcd}
\end{equation}
It is clear that $X_i \times_G F \cong X_i \times_X Y$ is affine. Since  $h_i : X_i \longrightarrow G$ is a Zariski open immersion relative to $\mathcal M$, it follows from Lemma $\ref{lema4.2}$ that the morphism $X_i \times_G F \longrightarrow F$ in \eqref{dig4.19} is a Zariski open immersion relative to $\mathcal M$. Further, by Lemma $\ref{grotop}$, the induced morphism 
\begin{equation} 
\coprod_{i \in I} X_i \times_G F \longrightarrow F
\end{equation}is an epimorphism in $Sh(Aff_{\mathcal C})_{\mathcal M}$. It follows that  $\{X_i \times_X Y \cong X_i \times_G F \longrightarrow F\}_{i \in I}$ is an affine Zariski $\mathcal M$-covering of $F$. This shows that $F$ is an $\mathcal M$-scheme.

\smallskip
(2) For each $i \in I$, we have a pullback square in $Sh(Aff_{\mathcal C})_{\mathcal M}$
\begin{equation}
\begin{tikzcd}
X_i \times_{F_0} F \arrow[rr, "g_{i}'"] \arrow[d] &  & F \arrow[d] \\
X_i \arrow[rr, "g_i"]                              &  & F_0        
\end{tikzcd}
\end{equation}By assumption, there exists an affine Zariski $\mathcal M$-covering $\{U_{i j} \longrightarrow X_i \times_{F_0} F\}_{j \in J_i}$ of $X_i \times_{F_0} F$. Since the morphism $g_i : X_i \longrightarrow F_0$ is a Zariski open immersion relative to $\mathcal M$, it follows from Lemma $\ref{lema4.2}$ that for each $j \in J_i$, the composition
\begin{equation}
 U_{ij} \longrightarrow X_i \times_{F_0} F \xrightarrow{\text{ }g_{i}'\text{ }} F
\end{equation}is a Zariski open immersion relative to $\mathcal M$. 

\smallskip
For each $i\in I$, since the morphism $\coprod_{j \in J_i} U_{i j} \longrightarrow X_i \times_{F_0} F$ induced by the affine Zariski $\mathcal M$-covering $\{U_{ij} \longrightarrow X_i \times_{F_0} F\}_{j \in J_i}$ is an epimorphism in $Sh(Aff_{\mathcal C})_{\mathcal M}$, the induced morphism
\begin{equation} \label{423dcf}
\coprod_{i \in I, j \in J_i} U_{i j} \longrightarrow \coprod_{i \in I} X_i \times_{F_0} F
\end{equation}is an epimorphism in $Sh(Aff_{\mathcal C})_{\mathcal M}$. Further, by Lemma $\ref{grotop}$, the induced morphism $\coprod_{i \in I} X_i \times_{F_0} F \xrightarrow{\text{ }\coprod_{i \in I} g_{i}'} F$ is an epimorphism in $Sh(Aff_{\mathcal C})_{\mathcal M}$.  It follows that $\{U_{ij} \longrightarrow X_i \times_{F_0} F \xrightarrow{\text{ }g_i'\text{ }} F\}_{i \in I, j \in J_i}$ is an affine Zariski $\mathcal M$-covering of $F$.  
\end{proof}
\begin{lem}
 Let $F$ be an $\mathcal M$-scheme and let $F_0 \hookrightarrow F$ be a Zariski open immersion relative to $\mathcal M$ in $Sh(Aff_{\mathcal C})_{\mathcal M}$. Then, $F_0$ is an $\mathcal M$-scheme.
\end{lem}
\begin{proof}
Since $F$ is an $\mathcal M$-scheme, there exists an affine Zariski $\mathcal M$-covering $\{g_i : X_i \longrightarrow F\}_{i\in I}$. For each $i \in I$, we have a pullback square in $Sh(Aff_{\mathcal C})_{\mathcal M}$
\begin{equation}
\begin{tikzcd}
F_0 \times_F X_i \arrow[rr, "p_i"] \arrow[d, "g_i'"] &  & X_i \arrow[d, "g_i"] \\
F_0 \arrow[rr, hook]                         &  & F                   
\end{tikzcd}
\end{equation}Since $F_0 \hookrightarrow F$ is a Zariski open immersion relative to $\mathcal M$, the morphism $p_i : F_0 \times_F X_i \longrightarrow X_i$ is a monomorphism whose image is Zariski open relative to $\mathcal M$ in $Sh(Aff_{\mathcal C})_{\mathcal M}$. By Definition \ref{zariskiopenforsheaves}(1), there is a family 
\begin{equation}
    \{h_{ij} : U_{ij} \longrightarrow X_i \}_{j\in J_i}
\end{equation}of spectral immersions relative to $\mathcal M$ where each $U_{ij} \in Aff_{\mathcal C}$ is affine and such that  $\coprod_{j \in J_i} U_{i j} \xrightarrow{\text{ }\coprod_{j \in J_i} h_{i j}\text{ }} X_i$ has image $F_0 \times_F X_i$ in $Sh(Aff_{\mathcal C})_{\mathcal M}$.  It follows that the family $\{U_{i j} \longrightarrow 
F_0\times_FX_i\longrightarrow  F_0 \}_{i \in I, j \in J_i}$ is an affine Zariski $\mathcal M$-covering of $F_0$. Hence, $F_0$ is an $\mathcal M$-scheme.  

\end{proof}
\begin{lem}\label{coprojzar}
    Let $\{F_i : i \in I\}$ be a family of objects in $Sh(Aff_{\mathcal C})_{\mathcal M}$ and let $F := \coprod_{i \in i} F_i$. Then, for each $i' \in I$, the canonical morphism $F_{i'} \longrightarrow F$ is a Zariski open immersion relative to $\mathcal M$.
\end{lem}
\begin{proof}
    Let $X = Spec (a) \in Aff_{\mathcal C}$ be an affine and
    \begin{equation}\label{eqtuni4.34} 
 f:   X \longrightarrow F = \coprod_{i \in I} F_i
    \end{equation}be a morphism in $Sh(Aff_{\mathcal C})_{\mathcal M}$. We have a pullback diagram in $Sh(Aff_{\mathcal C})_{\mathcal M}$
    \begin{equation} 
\xymatrix{
F_{i'} \times_F X \ar[rr] \ar[d] &  & X \ar[d]^f \\
F_{i'} \ar[rr]                      &  & F       \\   
}
    \end{equation}
    It is clear that $F_{i'} \longrightarrow F$ is a monomorphism in $Sh(Aff_{\mathcal C})_{\mathcal M}$. Hence,   $F_{i'} \times_F X \longrightarrow X$ is a monomorphism in $Sh(Aff_{\mathcal C})_{\mathcal M}$. We claim that its image is Zariski open relative to $\mathcal M$.
    
    \smallskip
    Under Yoneda Lemma, the morphism $f$  in \eqref{eqtuni4.34} corresponds to an element of $F(Spec(a))$. We note that the coproduct $F = \coprod_{i \in I} F_i$ in $Sh(Aff_{\mathcal C})_{\mathcal M}$ is the sheafification of the coproduct of the family $\{F_i\}_{i\in I}$ in $PSh(Aff_{\mathcal C})$. Using the   description of sheafification as a filtered colimit (see for instance, \cite[$\S$ 3.5]{MM}) there exists a spectral $\mathcal M$-cover $\{\alpha_k^{op} : X_k \longrightarrow X\}_{k \in K}$ such that for each $k\in K$, there is an element
    $u(k)\in I$   along with a  factorization
\begin{equation}
\begin{CD}
X_k @>\alpha_k^{op}>>  X   \\
@Vf_kVV @VVfV\\
F_{u(k)} @>>> F\\          
\end{CD}
\end{equation}
It follows that there exists  $g_k : X_k \longrightarrow F_{u(k)} \times_F X$ such that   the following diagram commutes for each $k\in K$
\begin{equation}
\begin{tikzcd}
       & X_k \arrow[d, "g_k", dashed] \arrow[ld, "f_k"'] \arrow[rd, "\alpha_k^{op}"] &   \\
F_{u(k)} & F_{u(k)} \times_F X \arrow[r] \arrow[l]       & X
\end{tikzcd}
\end{equation}Using Lemma \ref{grotop}, we have $X \cong F \times_F X \cong \coprod_{i \in I} (F_{i} \times_F X)$. Now since $\{\alpha_k^{op} : X_k \longrightarrow X\}_{k \in K}$ is a spectral $\mathcal M$-cover, it follows from Lemma \ref{coverinducepiintopos} that the induced morphism 
\begin{equation}
    \coprod_{i \in I} \coprod_{u(k)=i} X_k \cong \coprod_{k \in K} X_k \xrightarrow{\quad\coprod_{k \in K} \alpha_k^{op}\text{ }=\text{ }\coprod_{i \in I} \coprod_{u(k)=i} g_k\quad} X \cong \coprod_{i \in I} (F_i \times_F X)
\end{equation}is an epimorphism in $Sh(Aff_{\mathcal C})_{\mathcal M}$.  We note that the following is a pullback square in $Sh(Aff_{\mathcal C})_{\mathcal M}$
\begin{equation}
\begin{tikzcd}
\coprod_{u(k)=i'} X_k \arrow[rrr, "\coprod_{u(k)=i'} g_k"] \arrow[d]        &  &  & F_{i'} \times_F X \arrow[d]        \\
\coprod_{i \in I} \coprod_{u(k)=i} X_k \arrow[rrr, "\coprod_{i \in I} \coprod_{u(k)=i} g_k"] &  &  & \coprod_{i \in I} (F_i \times_F X) \cong X
\end{tikzcd}
\end{equation}Using Lemma \ref{grotop}, it follows that the image of $\coprod_{u(k)=i'} \alpha_k^{op} : \coprod_{u(k)=i'} X_{k} \longrightarrow X$ is $F_{i'} \times_F X$. This completes the proof.
\end{proof}

\begin{thm}\label{pullbacksofschemes}
    The subcategory $Sch(\mathcal C)_{\mathcal M}$ of $Sh(Aff_{\mathcal C})_{\mathcal M}$ is closed under

    \smallskip
    (1) coproducts.
    
    \smallskip
    (2) pullbacks.
\end{thm}
\begin{proof}
   (1) Let $\{F_i\}_{i\in I}$ be a family of $\mathcal M$-schemes. We need to show that the coproduct $F := \coprod_{i \in I} F_i$ in $Sh(Aff_{\mathcal C})_{\mathcal M}$ is an $\mathcal M$-scheme. For each $i \in I$, there exists an affine Zariski $\mathcal M$-covering $\{U_{ij} \longrightarrow F_i\}_{j \in J_i}$. Since the induced morphism $\coprod_{j \in J_i} U_{ij} \longrightarrow F_i$ is an epimorphism in $Sh(Aff_{\mathcal C})_{\mathcal M}$, the induced morphism
   \begin{equation} 
 \coprod_{i \in I}  \coprod_{j\in J_i} U_{ij} \longrightarrow F = \coprod_{i \in I} F_i
   \end{equation}is an epimorphism in $Sh(Aff_{\mathcal C})_{\mathcal M}$. Using Lemma \ref{coprojzar} and Lemma \ref{lema4.2}, it follows that $\{U_{ij} \longrightarrow F_i \longrightarrow F\}_{i \in I, j \in J_i}$ is an affine Zariski $\mathcal M$-covering for $F$. Hence, $F$ is an $\mathcal M$-scheme.

\medskip
(2) Let $F \xrightarrow{\text{ }f\text{ }} H \xleftarrow{\text{ }g\text{ }} G$ be morphisms in $Sch(\mathcal C)_{\mathcal M}$. We claim that $F \times_H G \in Sh(Aff_{\mathcal C})_{\mathcal M}$ is an $\mathcal M$-scheme. Since $F$ and $G$ are $\mathcal M$-schemes, there exist affine Zariski $\mathcal M$-coverings $\{X_i \longrightarrow F\}_{i \in I}$ and $\{Y_j \longrightarrow G\}_{j \in J}$. For each $i \in I$ and $j \in J$, we consider the following commutative diagram in $Sh(Aff_{\mathcal C})_{\mathcal M}$
\begin{equation}\label{diagrm4.2}   
\begin{tikzcd}
X_i \times_H Y_j \cong (X_i \times_H G) \times_G Y_j \arrow[d] \arrow[rr] &                                  & Y_j \arrow[d] \\
X_i \times_H G \cong X_i \times_F (F \times_H G) \arrow[r] \arrow[d]                 & F \times_H G \arrow[r, "f'"] \arrow[d, "g'"] & G \arrow[d, "g"]   \\
X_i \arrow[r]                                      & F \arrow[r, "f"]                      & H            
\end{tikzcd}
\end{equation} Using Lemma $\ref{lema4.5}(2)$, to show that $F \times_H G$ is an $\mathcal M$-scheme, it suffices to show that each $X_i \times_H G$ is an $\mathcal M$-scheme. By a further application of Lemma $\ref{lema4.5}(2)$ to the composition
\begin{equation} 
X_i \times_H G \longrightarrow F \times_H G \longrightarrow G
\end{equation}in \ref{diagrm4.2}, we see that it suffices to show that each $(X_i \times_H G) \times_G Y_j \cong X_i \times_H Y_j$ is an $\mathcal M$-scheme. Hence, we may assume that $F=X$ and $G=Y$ are affine. We will show that the pullback in $Sh(Aff_{\mathcal C})_{\mathcal M}$ of the diagram $X \xrightarrow{\text{ }f\text{ }} H \xleftarrow{\text{ }g\text{ }} Y$ is an $\mathcal M$-scheme.

\smallskip
Let $\{h_k : Z_k \longrightarrow H\}_{k \in K}$ be an affine Zariski $\mathcal M$-covering of the $\mathcal M$-scheme $H$. For each $k \in K$, we consider the following pullback diagram in $Sh(Aff_{\mathcal C})_{\mathcal M}$
\begin{equation}
\begin{tikzcd}
Z_k \times_H X \arrow[d, "h^X_k"'] \arrow[rr, "f_k"] &  & Z_k \arrow[d, "h_k"] &  & Z_k \times_H Y \arrow[d, "h^Y_k"] \arrow[ll, "g_k"'] \\
X \arrow[rr, "f"]                                    &  & H                    &  & Y \arrow[ll, "g"']                                  
\end{tikzcd}
\end{equation}Since $h_k : Z_k \longrightarrow H$ is a Zariski open immersion relative to $\mathcal M$, we see that  $h_k^X : Z_k \times_H X \longrightarrow X$ and $h_k^Y : Z_k \times_H Y \longrightarrow Y$ are monomorphisms whose images are Zariski open relative to $\mathcal M$. By definition, there exists a family 
\begin{equation}
\{U_{ki} \longrightarrow X\}_{i \in I_k}\qquad (\text{resp. }\{V_{kj} \longrightarrow Y\}_{j \in J_k})
\end{equation}of spectral immersions relative to $\mathcal M$ such that the image of the canonical morphism $\coprod_{i \in I_k} U_{ki} \longrightarrow X$ (resp. $\coprod_{j \in J_k} V_{kj} \longrightarrow Y$) is $Z_k \times_H X$ (resp. $Z_k \times_H Y$). Using Lemma $\ref{grotop}$, it follows that the induced morphisms
\begin{equation}
  \coprod_{k \in K, i \in I_k} U_{ki} \longrightarrow \coprod_{k \in K} Z_k \times_H X \xrightarrow{\text{ }\coprod_{k \in K} h_k^X\text{ }} X\quad\text{and}\quad\coprod_{k \in K, j \in J_k} V_{kj} \longrightarrow \coprod_{k \in K} Z_k \times_H Y \xrightarrow{\text{ }\coprod_{k \in K} h_k^Y\text{ }} Y
\end{equation}are epimorphisms in $Sh(Aff_{\mathcal C})_{\mathcal M}$. Hence using Proposition \ref{oneimpliesother}, $\{U_{ki} \longrightarrow X\}_{k \in K, i \in I_k}$ (resp. $\{V_{kj} \longrightarrow Y\}_{k \in K, j \in J_k}$) is an affine Zariski $\mathcal M$-covering of $X$ (resp. $Y$). Further, for each $k \in K, i' \in I_k$ and $j' \in J_k$, we have a commutative diagram
\begin{equation}\label{438uy}
\begin{tikzcd}
{U_{ki'}} \arrow[d] \arrow[r] & {\coprod_{i \in I_k}U_{ki}} \arrow[r, two heads] & Z_k \times_H X \arrow[r, "f_k"] & Z_k \arrow[d, "h_k"] & Z_k \times_H Y \arrow[l, "g_k"'] & {\coprod_{j \in J_k} V_{kj}} \arrow[l, two heads] & {V_{kj'}} \arrow[l] \arrow[d] \\
X \arrow[rrr, "f"]              &                                         &                          & H             &                          &                                          & Y \arrow[lll, "g"']            
\end{tikzcd}
\end{equation} We now consider the following commutative diagram in $Sh(Aff_{\mathcal C})_{\mathcal M}$
\begin{equation}
\begin{tikzcd}
{U_{ki'} \times_H V_{kj'} \cong (U_{ki'} \times_H Y) \times_Y V_{kj'}} \arrow[rr] \arrow[d] &                                  & {V_{kj'}} \arrow[d] \arrow[rddd, shift left] &                      \\
{U_{ki'} \times_H Y \cong U_{ki'} \times_X (X \times_H Y)} \arrow[r] \arrow[d]                      & X \times_H Y \arrow[r, "f'"] \arrow[d, "g'"] & Y \arrow[d, "g"]                                   &                      \\
{U_{ki'}} \arrow[r] \arrow[rrrd, shift right]                 & X \arrow[r, "f"]                      & H                                             &                      \\
                                                               &                                  &                                               & Z_k \arrow[lu, "h_k"]
\end{tikzcd}
\end{equation}where the morphism $U_{ki'} \longrightarrow Z_k$ (resp. $V_{kj'} \longrightarrow Z_k$) is the composition
\begin{equation}
    U_{ki'} \longrightarrow \coprod_{i \in I_k} U_{ki} \twoheadrightarrow Z_k \times_H X \xrightarrow{\text{ }f_k\text{ }} Z_k\qquad(\text{ resp. }V_{kj'} \longrightarrow \coprod_{j \in J_k} V_{kj} \twoheadrightarrow Z_k \times_H Y \xrightarrow{\text{ }g_k\text{ }} Z_k)
\end{equation} from \eqref{438uy}. Using Lemma \ref{lema4.5}(2), to show that $X \times_H Y$ is an $\mathcal M$-scheme, it suffices to show that each $(U_{ki'} \times_H Y) \times_Y V_{kj'} \cong U_{ki'} \times_H V_{kj'}$ is an $\mathcal M$-scheme. Using Lemma \eqref{zaropenimmono}, we note that $h_k : Z_k \longrightarrow H$ is a monomorphism in $Sh(Aff_{\mathcal C})_{\mathcal M}$. Hence, $U_{ki'} \times_H V_{kj'} \cong U_{ki'} \times_{Z_k} V_{kj'}$, which is an affine scheme. This completes the proof.
\end{proof}
\begin{lem}\label{factoexplicit}
    Let $Y = Spec(a) \in Aff_{\mathcal C}$ be an affine, $F \in Sh(Aff_{\mathcal C})_{\mathcal M}$ and $\{H_i\}_{i \in I}$ be a family of objects in $Sh(Aff_{\mathcal C})_{\mathcal M}$ with $H := \coprod_{i \in I} H_i \in Sh(Aff_{\mathcal C})_{\mathcal M}$. Let $f : H \longrightarrow F$ be an epimorphism and $g : Y \longrightarrow F$ be a morphism in $Sh(Aff_{\mathcal C})_{\mathcal M}$. Then, there exists a family $\{\alpha_k^{op} : Y_k \longrightarrow Y\}_{k \in K}$ of spectral immersions relative to $\mathcal M$ such that 
    
\smallskip
(i) the morphism $\coprod_{k \in K} \alpha_k^{op} : \coprod_{k \in K} Y_k \longrightarrow Y$ is an epimorphism in $Sh(Aff_{\mathcal C})_{\mathcal M}$.

\smallskip
(ii) for each $k \in K$, there is an element $w(k) \in I$ along with a factorization
    \begin{equation}
\begin{tikzcd}
Y_k \arrow[rr, "\alpha_k^{op}"] \arrow[d, dashed] &                                       & Y \arrow[d, "g"] \\
H_{w(k)} \arrow[r]               & H = \coprod_{i \in I} H_i \arrow[r, "f"'] & F               
\end{tikzcd}
    \end{equation}
\end{lem}
\begin{proof}
  Applying Yoneda Lemma, the morphism $g$ corresponds to an element of $F(Spec(a))$. Since $f : H \longrightarrow F$ is an epimorphism, it follows from \cite[$\S$ 3.7, Corollary 5]{MM} that there is a spectral $\mathcal M$-cover $\{\beta_j^{op} : V_j \longrightarrow Y\}_{j \in J}$ such that for each $j \in J$, there is a factorization
    \begin{equation}
\begin{tikzcd}
V_j \arrow[r, "\beta_j^{op}"] \arrow[d, "g_j"', dashed] & Y \arrow[d, "g"] \\
\coprod_{i \in I} H_i = H \arrow[r, "f"]                     & F          
\end{tikzcd}
    \end{equation}For each $i \in I$ and $j \in J$, we consider the following pullback square in $Sh(Aff_{\mathcal C})_{\mathcal M}$
\begin{equation}\label{digirm4.43}
\begin{tikzcd}
H_i \times_H V_j \arrow[rr] \arrow[d] &  & V_j \arrow[d, "g_j"]  \\
H_i \arrow[rr]                                              &  & H
\end{tikzcd}
\end{equation}
Using Lemma \ref{coprojzar}, we note that the morphism $H_i \times_H V_j \longrightarrow V_j$ in \eqref{digirm4.43} is a monomorphism whose image is Zariski open relative to $\mathcal M$ in $Sh(Aff_{\mathcal C})_{\mathcal M}$. For $i\in I$, $j\in J$, by choosing spectral immersions $\{Y_k^{ij}\longrightarrow V_j\}_{k\in K_{ij}}$  whose image is the same as that of $H_i \times_H V_j \longrightarrow V_j$, the result now follows from Lemma \ref{stabilityspec}, Lemma \ref{coverinducepiintopos} and Lemma \ref{grotop}.
\end{proof}
The next proposition gives a characterization of the subcategory $Sch(\mathcal C)_{\mathcal M}$ of $Sh(Aff_{\mathcal C})_{\mathcal M}$ in terms of quotients.
\begin{Thm}\label{equivalencerel}
An object $F \in Sh(Aff_{\mathcal C})_{\mathcal M}$ is an $\mathcal M$-scheme if and only if there exists a family $\{U_i\}_{i \in I}$ of affines and an equivalence relation $R \hookrightarrow H \times H$ on $H := \coprod_{i \in I} U_i$ in $Sh(Aff_{\mathcal C})_{\mathcal M}$ such that

\smallskip
(i) For any $i, j \in I$, the composition $R_{i, j} := R \times_{H \times H} (U_i \times U_j) \xrightarrow{\text{ }r_{i, j}\text{ }} U_i \times U_j \longrightarrow U_i$ is a Zariski open immersion relative to $\mathcal M$.

\smallskip
(ii) For each $i \in I$, the subobject $R_{i, i} \hookrightarrow U_i \times U_i$ is the image of the diagonal morphism $U_i \xrightarrow{\text{ }(1_{U_i}, 1_{U_i})\text{ }} U_i \times U_i$.

\smallskip
(iii) $F \cong H/R$, i.e
\begin{equation}\label{eqt4.5}
    F = Coeq\left(
\begin{tikzcd}
R \arrow[r, hook] & H \times H \arrow[r, "\pi_1", shift left=2] \arrow[r, "\pi_2"', shift right=2] & H
\end{tikzcd} \right)
\end{equation}in $Sh(Aff_{\mathcal C})_{\mathcal M}$, where $\pi_1, \pi_2 : H \times H \longrightarrow H$ are the two projections.
\end{Thm}
\begin{proof}
    Suppose that $F$ is an $\mathcal M$-scheme. Then, $F$ has an affine Zariski $\mathcal M$-covering $\{f_i : U_i \longrightarrow F\}_{i\in I}$. We set $H := \coprod_{i \in I} U_i$. Let $R$ be the following pullback in $Sh(Aff_{\mathcal C})_{\mathcal M}$
    \begin{equation}
\begin{tikzcd}
R := H \times_F H \arrow[rr, "p_1"] \arrow[d, "p_2"'] &  & H= \coprod_{i \in I} U_i \arrow[d, "\coprod_{i \in I} f_i", two heads] \\
H = \coprod_{i \in I} U_i \arrow[rr, "\coprod_{i \in I} f_i", two heads]                                   &  & F                     
\end{tikzcd}
    \end{equation}It may be verified that $R \xrightarrow{\text{ }(p_1, p_2)\text{ }} H \times H$ is an equivalence relation on $H$. For each $i, j \in I$, let $R_{i, j} := R \times_{H \times H}(U_i \times U_j) \xrightarrow{\text{ }r_{i, j}\text{ }} U_i \times U_j$ denote the canonical morphism to $U_i \times U_j$. It may be verified that $R_{i, j}$ may also be written as the pullback
\begin{equation}\label{alterpullback}
\begin{tikzcd}
U_i \times_F U_j \cong {R_{i, j}} \arrow[r, "r_{i, j}"] \arrow[d, "r_{i, j}"'] & U_i \times U_j \arrow[r] & U_j \arrow[dd, "f_j"] \\
U_i \times U_j \arrow[d]       &                          &                       \\
U_i \arrow[rr, "f_i"']         &                          & F                    
\end{tikzcd}
\end{equation}Since $f_j : U_j \longrightarrow F$ is a Zariski open immersion relative to $\mathcal M$, it follows from Lemma $\ref{lema4.2}$ that the composition $R_{i, j} \xrightarrow{\text{ }r_{i, j}\text{ }} U_i \times U_j \longrightarrow U_i$ in \eqref{alterpullback} is a Zariski open immersion relative to $\mathcal M$. Hence, $R$ satisfies condition (i). Taking $i = j$ in \ref{alterpullback}, we see that there is a morphism $s_i : U_i \longrightarrow R_{i, i}$ such that
\begin{equation}
   \left(U_i \xrightarrow{\text{ }(1_{U_i}, 1_{U_i})\text{ }} U_i \times U_i\right) =  \left(U_i \xrightarrow{\text{ }s_i\text{ }} R_{i, i} \xrightarrow{\text{ }r_{i, i}\text{ }} U_i \times U_i\right)
\end{equation}Using Lemma \ref{zaropenimmono}, it may be verified that $s_i$ is an isomorphism. Hence, $R$ satisifies condition (ii). Since 
$H\longrightarrow F$ is an epimorphism, condition (iii) follows from \cite[$\S$ 4.7, Theorem 8]{MM}.

    \smallskip
    Conversely, let $\{U_i\}_{i \in I}$ be a family of affines and $R \hookrightarrow H \times H$ be an equivalence relation on $H = \coprod_{i \in I} U_i$ satisfying conditions (i), (ii) and (iii). Let $f : H \longrightarrow F$ be the canonical epimorphism associated to the coequalizer in $\eqref{eqt4.5}$. For each $i \in I$, let $f_i : U_i \longrightarrow F$ be the composition $U_i \longrightarrow H \xrightarrow{\text{ }f\text{ }} F$. Using condition (iii), it may be verified that for each $i, j \in I$, we have a pullback square  as in \eqref{alterpullback} in $Sh(Aff_{\mathcal C})_{\mathcal M}$. We claim that
    \begin{equation}
      \{f_i : U_i \longrightarrow F\}_{i \in I}
    \end{equation}is an affine Zariski $\mathcal M$-covering. Since $f : H= \coprod_{i \in I} U_i  \longrightarrow F$ is an epimorphism, it suffices to show   that  $f_i : U_i \longrightarrow F$ is a Zariski open immersion relative to $\mathcal M$ for each $i\in I$. Since the monomorphism $U_i \xrightarrow{\text{ }(1_{U_i}, 1_{U_i})\text{ }} U_i \times U_i$ has a retraction, it follows from condition (ii) that 
    \begin{equation}
      U_i \cong R_{i, i} \cong U_i \times_F U_i
    \end{equation}Hence, $f_i : U_i \longrightarrow F$ is a monomorphism.
    
    \smallskip
    Let $Y = Spec(a) \in Aff_{\mathcal C}$ be affine and $g : Y \longrightarrow F$ be a morphism in $Sh(Aff_{\mathcal C})_{\mathcal M}$. We consider the following pullback square in $Sh(Aff_{\mathcal C})_{\mathcal M}$
    \begin{equation}\label{digirm4.48}
\begin{tikzcd}
U_i \times_F Y \arrow[rr] \arrow[d] &  & Y \arrow[d, "g"] \\
U_i \arrow[rr, "f_i"']              &  & F               
\end{tikzcd}
    \end{equation}
    It is clear that the morphism $U_i \times_F Y \longrightarrow Y$ in \eqref{digirm4.48} is a monomorphism. We claim that its image is Zariski open relative to $\mathcal M$. Using Lemma \ref{factoexplicit}, it follows that there is a family $\{\alpha_k^{op} : Y_k \longrightarrow Y\}_{k\in K}$ of spectral immersions relative to $\mathcal M$ such that the morphism $\coprod_{k \in K} \alpha_k^{op} : \coprod_{k \in K} Y_k \longrightarrow Y$ is an epimorphism in $Sh(Aff_{\mathcal C})_{\mathcal M}$ and such that for each $k \in K$, there is an element $w(k) \in I$ along with a factorization
    \begin{equation}\label{eqton4.49}
\begin{tikzcd}
Y_k \arrow[rr, "\alpha_k^{op}"] \arrow[d, "h_k"', dashed] &  & Y \arrow[d, "g"] \\
U_{w(k)} \arrow[rr, "f_{w(k)}"']                          &  & F               
\end{tikzcd}
    \end{equation}Since the square
\begin{equation}
\begin{tikzcd}
{R_{i, w(k)}} \arrow[r, "r_{i,w(k)}"] \arrow[d, "r_{i, w(k)}"'] & U_i \times U_{w(k)} \arrow[r] & U_{w(k)} \arrow[dd, "f_{w(k)}"] \\
U_i \times U_{w(k)} \arrow[d]       &                          &                       \\
U_i \arrow[rr, "f_i"']         &                          & F                    
\end{tikzcd}
\end{equation}is a pullback in $Sh(Aff_{\mathcal C})_{\mathcal M}$, it follows from $\eqref{eqton4.49}$ that
\begin{equation}
    (U_i \times_F Y) \times_Y Y_k \cong (U_i \times_F U_{w(k)}) \times_{U_{w(k)}} Y_k \cong R_{i, w(k)} \times_{U_{w(k)}} Y_k
\end{equation} We now have the following pullback square in $Sh(Aff_{\mathcal C})_{\mathcal M}$
\begin{equation}
\begin{tikzcd}\label{eqton4.52}
{(U_i \times_F Y) \times_Y Y_k \cong R_{i, w(k)} \times_{U_{w(k)}} Y_k} \arrow[rrrr] \arrow[d] &  &                                &  & Y_k \arrow[d, "h_k"] \\
{R_{i, w(k)}} \arrow[rr, "r_{i, w(k)}"']                                                                       &  & U_i \times U_{w(k)} \arrow[rr] &  & U_{w(k)}     
\end{tikzcd}
\end{equation}Since by condition (i), the morphism $R_{i, w(k)} \xrightarrow{\text{ }r_{i, w(k)}\text{ }} U_i \times U_{w(k)} \longrightarrow U_{w(k)}$ is a Zariski open immersion relative to $\mathcal M$, it follows that the morphism 
$
(U_i \times_F Y) \times_Y Y_k \cong R_{i, w(k)} \times_{U_{w(k)}} Y_k \longrightarrow Y_k
$ in \eqref{eqton4.52} is a monomorphism whose image is Zariski open relative to $\mathcal M$. Hence, there exists a family $\{Z_{ikj} \longrightarrow Y_k\}_{j \in J_{ik}}$ of spectral immersions relative to $\mathcal M$ such that the image of the induced morphism $\coprod_{j \in J_{ik}} Z_{ikj} \longrightarrow Y_k$ is $(U_i \times_F Y) \times_Y Y_k$. It follows from Lemma \ref{stabilityspec} and Lemma $\ref{grotop}$ that $\{Z_{ikj} \longrightarrow Y_k \xrightarrow{\alpha_k^{op}} Y\}_{k \in K, j \in J_{ik}}$ is a family of spectral immersions relative to $\mathcal M$ such that the image of the induced morphism $\coprod_{k \in K, j \in J_{ik}} Z_{ikj} \longrightarrow Y$ is $U_i \times_F Y$. Hence, $U_i \times_F Y$ is Zariski open relative to $\mathcal M$. This shows that $f_i : U_i \longrightarrow F$ is a Zariski open immersion relative to $\mathcal M$. The proof is now complete.
\end{proof}
\section{Change of Base}
  Throughout this section, we assume that $(\mathcal C, \otimes, 1)$ and $(\mathcal D, \otimes, 1)$ are closed symmetric monoidal categories which are both complete and cocomplete. We fix an adjunction
  \begin{equation}\label{monadjoint}
      \left(\mathbf B :  \mathcal C \longrightarrow \mathcal D,\text{ }\mathbf A : \mathcal D \longrightarrow \mathcal C\right)
  \end{equation}with unit $\eta:id_{\mathcal C}\longrightarrow \mathbf A\mathbf B$ and counit $\varepsilon:\mathbf B\mathbf A\longrightarrow id_{\mathcal D}$. We further assume that the left adjoint $\mathbf B$ has a strong symmetric monoidal struture i.e. there is a natural isomorphism
  \begin{equation}
  \left(\mathbf B(a) \otimes \mathbf B(b) \xrightarrow{\text{ }\sim\text{ }} \mathbf B(a \otimes b) : a, b \in \mathcal C\right)\\
  \end{equation} and an isomorphism $1 \xrightarrow{\text{ }\sim\text{ }} \mathbf B(1)$ subject to certain coherence axioms (see for instance, \cite[$\S$ 11.2]{Mac}). In that case, it may be verified that the right adjoint $\mathbf A$ has a lax symmetric monoidal structure given by maps  (see for instance, \cite[$\S$ 11.2]{Mac})
  \begin{equation} \left(\mathbf A(a) \otimes \mathbf A(b) \longrightarrow  \mathbf A(a \otimes b) : a, b \in \mathcal D\right)
  \end{equation} It then follows that the adjunction in $\eqref{monadjoint}$ induces an adjunction
  \begin{equation}\label{eqtunio5.2}
      \left(\mathbf B : Comm(\mathcal C) \longrightarrow Comm(\mathcal D),\text{ }\mathbf A : Comm(\mathcal D) \longrightarrow Comm(\mathcal C)\right)
  \end{equation}or equivalently, an adjunction
  \begin{equation}\label{eqtunio5.3}
      \left(\mathbf A^{op} : Aff_{\mathcal D} \longrightarrow Aff_{\mathcal C},\text{ }\mathbf B^{op} : Aff_{\mathcal C} \longrightarrow Aff_{\mathcal D}\right)
  \end{equation}Hence, there is an adjunction $(\mathbf A_!, \mathbf B_!)$
  \begin{equation}\label{equt5.7}
      \mathbf A_! = \_\_ \circ \mathbf A : PSh(Aff_{\mathcal C}) \longrightarrow PSh(Aff_{\mathcal D}),\text{ }\mathbf B_! = \_\_ \circ \mathbf B : PSh(Aff_{\mathcal D}) \longrightarrow PSh(Aff_{\mathcal C})
  \end{equation}
We also fix a left $\mathcal C$-actegory $(\mathcal M, \boxtimes)$ (resp. a left $\mathcal D$-actegory $(\mathcal N, \boxplus)$) such that $\mathcal M$ (resp. $\mathcal N$) is both complete and cocomplete and the bifunctor $\boxtimes : \mathcal C \times \mathcal M \longrightarrow \mathcal M$ (resp. $\boxplus : \mathcal D \times \mathcal N \longrightarrow \mathcal N$) preserves colimits in both variables. 
\begin{lem}\label{lemcont}
    (1) The functor $\mathbf B^{op} : Aff_{\mathcal C} \longrightarrow Aff_{\mathcal D}$ preserves limits.

    \smallskip
    (2) Suppose that $\mathbf B^{op} : Aff_{\mathcal C} \longrightarrow Aff_{\mathcal D}$ takes fpqc $\mathcal M$-covers to fpqc $\mathcal N$-covers and $\mathbf A : \mathcal D \longrightarrow \mathcal C$ preserves filtered colimits.  Then $\mathbf B^{op} : Aff_{\mathcal C} \longrightarrow Aff_{\mathcal D}$ takes spectral immersions relative to $\mathcal M$ to spectral immersions relative to $\mathcal N$ and takes spectral $\mathcal M$-covers to spectral $\mathcal N$-covers.
\end{lem}
\begin{proof}
(1) The result is clear from the adjunction in \eqref{eqtunio5.3}.

\smallskip
(2) It follows from \eqref{eqtunio5.2} that $\mathbf B : Comm(\mathcal C) \longrightarrow Comm(\mathcal D)$ preserves colimits and hence epimorphisms. Let $\alpha : a \longrightarrow b$ be an $\mathcal M$-flat morphism in $Comm(\mathcal C)$. Then $\{\alpha^{op}:Spec(b)\longrightarrow Spec(a), 1_{Spec(a)}:Spec(a)\longrightarrow Spec(a)\}$ is an fpqc $\mathcal M$-cover in $Aff_{\mathcal C}$. Since $\mathbf B^{op}$ takes fpqc $\mathcal M$-covers to fpqc $\mathcal N$-covers, we see that $ \{\mathbf B^{op}(\alpha^{op})=\mathbf B(\alpha)^{op}, 1_{Spec(\mathbf B(a))}\}$ 
 is an fpqc $\mathcal N$-cover in $Aff_{\mathcal D}$. In particular, $\mathbf B(\alpha)$ is $\mathcal N$-flat, i.e., $\mathbf B : Comm(\mathcal C) \longrightarrow Comm(\mathcal D)$ takes $\mathcal M$-flat morphisms to $\mathcal N$-flat morphisms. 

\smallskip
We now show that $\mathbf B : Comm(\mathcal C) \longrightarrow Comm(\mathcal D)$ preserves morphisms of finite type. We note that since $\mathbf A : \mathcal D \longrightarrow \mathcal C$ preserves filtered colimits, so does the induced functor $\mathbf A : Comm(\mathcal D) \longrightarrow Comm(\mathcal C)$. Let $\alpha : a \longrightarrow b$ be a morphism of finite type in $Comm(\mathcal C)$. We need to show that $\mathbf B(\alpha) : \mathbf B(a) \longrightarrow \mathbf B(b)$ is of finite type. Let $\mathbb D : \mathcal I \longrightarrow \mathbf B(a)/Comm(\mathcal D)$ be a filtered diagram. It may be verified that \eqref{eqtunio5.2} induces an adjunction $( \widehat{\mathbf B} , \widehat{\mathbf A} )$
\begin{equation}
    \begin{split}
       \widehat{\mathbf B} : a/Comm(\mathcal C) \longrightarrow \mathbf B(a)/Comm(\mathcal D)& \qquad   (a \xrightarrow{\text{ }\beta\text{ }} c) \mapsto ( \mathbf B(a) \xrightarrow{\text{ }\mathbf B(\beta)\text{ }} \mathbf B(c)) \\ 
        \widehat{\mathbf A} : \mathbf B(a)/Comm(\mathcal D) \longrightarrow a/Comm(\mathcal C)&\qquad  ( \mathbf B(a) \xrightarrow{\text{ }\gamma\text{ }} d) \mapsto (a \xrightarrow{\text{ }\eta_a\text{ }} \mathbf A\mathbf B(a) \xrightarrow{\text{ }\mathbf A(\gamma)\text{ }} \mathbf A(d))  
    \end{split}
\end{equation}We note that we have a commutative square
\begin{equation}\label{58y}
\begin{tikzcd}
\mathbf B(a)/Comm(\mathcal D) \arrow[rr, "\widehat{\mathbf A}"] \arrow[d] &  & a/Comm(\mathcal C) \arrow[d] \\
Comm(\mathcal D) \arrow[rr, "\mathbf A"']                                                      &  & Comm(\mathcal C)                                
\end{tikzcd}
\end{equation} where the vertical arrows in \eqref{58y} are the forgetful functors. Since $\mathbf A : Comm(\mathcal D) \longrightarrow Comm(\mathcal C)$ and the forgetful functors 
in \eqref{58y} preserve   filtered colimits and the   functor   $a/Comm(\mathcal C) \longrightarrow \mathcal C$ is conservative, it follows from \eqref{58y} that $\widehat{\mathbf A} : \mathbf B(a)/Comm(\mathcal D) \longrightarrow a/Comm(\mathcal C)$ preserves filtered colimits. Since $\alpha : a \longrightarrow b$ is a morphism of finite type in $Comm(\mathcal C)$, we have
\begin{equation}
\begin{split}
    \mathbf B(a)/Comm(\mathcal D)(  \mathbf B(\alpha), \underset{i \in \mathcal I}{colim}\text{ }\mathbb D(i)) &\cong \mathbf B(a)/Comm(\mathcal D)(\widehat{\mathbf B}(\alpha), \underset{i \in I}{colim}\text{ }\mathbb D(i))\\
    &\cong a/Comm(\mathcal C)( \alpha, \widehat{\mathbf A}(\underset{i \in I}{colim}\text{ }\mathbb D(i)))\\
     &\cong a/Comm(\mathcal C)( \alpha, \underset{i \in I}{colim}\text{ }\widehat{\mathbf A}(\mathbb D(i)))\\
     &\cong \underset{i \in I}{colim}\text{ }a/Comm(\mathcal C)(\alpha, \widehat{\mathbf A}(\mathbb D(i)))\\
     &\cong \underset{i \in I}{colim}\text{ }\mathbf B(a)/Comm(\mathcal D)(\widehat{\mathbf B}(\alpha), \mathbb D(i))\\
     &\cong \underset{i \in I}{colim}\text{ }\mathbf B(a)/Comm(\mathcal D)(  \mathbf B(\alpha), \mathbb D(i))
\end{split}
\end{equation}
  Hence, $\mathbf B : Comm(\mathcal C) \longrightarrow Comm(\mathcal D)$ preserves morphisms of finite type. It follows that $\mathbf B^{op} : Aff_{\mathcal C} \longrightarrow Aff_{\mathcal D}$ takes spectral immersions relative to $\mathcal M$ to spectral immersions relative to $\mathcal N$. Finally, since $\mathbf B^{op} : Aff_{\mathcal C} \longrightarrow Aff_{\mathcal D}$ takes fpqc $\mathcal M$-covers to fpqc $\mathcal N$-covers, it follows that $\mathbf B^{op}$ also takes spectral $\mathcal M$-covers to spectral $\mathcal N$-covers. This concludes the proof.
\end{proof}
Since the functor $\mathbf B : \mathcal C \longrightarrow \mathcal D$ has a strong symmetric monoidal structure, we now recall (see for instance, \cite[Proposition 3.6.1]{act}) that the left $\mathcal D$-action $\boxplus : \mathcal D \times \mathcal N \longrightarrow \mathcal N$ restricts to the left $\mathcal C$-action $\boxplus^{\mathbf B} : \mathcal C \times \mathcal N \longrightarrow \mathcal N$ given by
  \begin{equation}\label{restricting}
      \mathcal C \times \mathcal N \xrightarrow{\text{ }\mathbf B \times 1_{\mathcal N}\text{ }} \mathcal D \times \mathcal N \xrightarrow{\text{ }\boxplus\text{ }} \mathcal N 
  \end{equation}
  We will denote the left $\mathcal C$-actegory $(\mathcal N, \boxplus^{\mathbf B})$ by $\mathbf B_*(\mathcal N)$. Since $\mathbf B : \mathcal C \longrightarrow \mathcal D$ preserves colimits and $\boxplus : \mathcal D \times \mathcal N \longrightarrow \mathcal N$ preserves colimits in both variables, it follows that $\boxplus^{\mathbf B} : \mathcal C \times \mathcal N \longrightarrow \mathcal N$ preserves colimits in both variables. Further, for every $a \in Comm(\mathcal C)$, we have $\mathbf{Mod}_{\mathbf B_*(\mathcal N)}(a) = \mathbf{Mod}_{\mathcal N}(\mathbf B(a))$. We now fix a lax $\mathcal C$-linear functor
\begin{equation}\label{actadjunct}
   (\mathbb L, \Gamma) : \mathbf B_*(\mathcal N) = (\mathcal N, \boxplus^{\mathbf B}) \longrightarrow (\mathcal M, \boxtimes)
\end{equation}We recall (see for instance, \cite[Lemma 3.3.7]{act}) that for each $a \in Comm(\mathcal C)$, there is an induced functor
\begin{equation}\label{modadjunct}
 \begin{split}
\mathbb L^a : \mathbf{Mod}_{\mathcal N}(\mathbf B(a)) = \mathbf{Mod}_{\mathbf B_*(\mathcal N)}(a) &\longrightarrow \mathbf{Mod}_{\mathcal M}(a)\\
\left(n, a \boxplus^{\mathbf B} n \xrightarrow{\text{ }\rho\text{ }} n\right) &\mapsto \left(\mathbb L(n), a \boxtimes \mathbb L(n) \xrightarrow{\text{ }\Gamma_{a, n}\text{ }} \mathbb L(a \boxplus^{\mathbf B} n) \xrightarrow{\text{ }\mathbb L(\rho)\text{ }} \mathbb L(n)\right)
\end{split}
\end{equation}By   abuse of notation, we will typically denote $\mathbb L^a$ simply by $\mathbb L$.
\begin{lem}\label{crucial}
Let $\alpha : a = (a, \mu_a, \iota_a) \longrightarrow (b, \mu_b, \iota_b) = b$ be a morphism in $Comm(\mathcal C)$. Then,

\smallskip
(1) For each $n=  (n, \rho) \in \mathbf{Mod}_{\mathcal N}(\mathbf B(a))$, there is a canonical morphism
    \begin{equation}\label{canonic}
       \theta_{\alpha, \mathbf B, \mathbb L, n} : b \boxtimes_a \mathbb L(n) \longrightarrow \mathbb L(\mathbf B(b) \boxplus_{\mathbf B(a)} n) 
    \end{equation}in $\mathbf{Mod}_{\mathcal M}(b)$.

\smallskip
(2) If $\mathbb L$ preserves coequalizers and $(\mathbb L, \Gamma) : \mathbf B_*(\mathcal N) \longrightarrow (\mathcal M, \boxtimes)$ is strong $\mathcal C$-linear, i.e., $\Gamma$ is a natural isomorphism, then for each $(n, \rho) \in \mathbf{Mod}_{\mathcal N}(\mathbf B(a))$, the morphism in $\eqref{canonic}$ is an isomorphism in $\mathbf{Mod}_{\mathcal M}(b)$.
\end{lem}
\begin{proof}
(1) We note that as in \eqref{coeq}, $\mathbf B(b) \boxplus_{\mathbf B(a)} n$ is given by the following coequalizer in $\mathbf{Mod}_{\mathcal N}(\mathbf B(b)) = \mathbf{Mod}_{\mathbf B_*(\mathcal N)}(b)$
\begin{equation}\label{eqtunio5.13}\small
\begin{tikzcd}
\mathbf B(b) \boxplus (\mathbf B(a) \boxplus n) \arrow[rrrr, "1_{\mathbf B(b)} \boxplus \rho", shift left=3] \arrow[rr, "\sim"', shift right] &  & (\mathbf B(b) \otimes \mathbf B(a)) \boxplus n \arrow[rr, "\psi"', shift right] &  & \mathbf B(b) \boxplus n \arrow[rrrr, "{coeq_{\mathbf B(\alpha), (n, \rho)}}"] &  &  &  & \mathbf B(b) \boxplus_{\mathbf B(a)} n
\end{tikzcd}
\end{equation}where $\psi$ is the composition
\begin{equation}
    (\mathbf B(b) \otimes \mathbf B(a)) \boxplus n \cong \mathbf B(b \otimes a) \boxplus n \xrightarrow{\text{ }\mathbf B(\mu_b \circ (1_b \otimes \alpha)) \boxplus 1_n\text{ }} \mathbf B(b) \boxplus n
\end{equation} Applying $\mathbb L : \mathbf{Mod}_{\mathcal N}(\mathbf B(b)) = \mathbf{Mod}_{\mathbf B_*(\mathcal N)}(b) \longrightarrow \mathbf{Mod}_{\mathcal M}(b)$ to \eqref{eqtunio5.13} gives the following commutative diagram
\begin{equation}\label{eqtuniq5.16}
\begin{tikzcd}
\mathbb L((\mathbf B(b) \otimes \mathbf B(a)) \boxplus n) \arrow[d, "\mathbb L(\psi)"']              &  & \mathbb L(\mathbf B(b) \boxplus (\mathbf B(a) \boxplus n)) \arrow[ll, "\sim"'] \arrow[rr, "\mathbb L(1_{\mathbf B(b)} \boxplus \rho)"] &  & \mathbb L(\mathbf B(b) \boxplus n) \arrow[d, "{\mathbb L(coeq_{\mathbf B(\alpha), (n, \rho)})}"] \\
\mathbb L(\mathbf B(b) \boxplus n) \arrow[rrrr, "{\mathbb L(coeq_{\mathbf B(\alpha), (n, \rho)})}"'] &  &                                                                                                                                        &  & \mathbb L(\mathbf B(b) \boxplus_{\mathbf B(a)} n)                                               
\end{tikzcd}
\end{equation}in $\mathbf{Mod}_{\mathcal M}(b)$. It may be verified that
\begin{equation} 
\Gamma_{b, \mathbf B(a) \boxplus n} \circ (1_b \boxtimes \Gamma_{a, n}) : b \boxtimes (a \boxtimes \mathbb L(n)) \longrightarrow \mathbb L(\mathbf B(b) \boxplus (\mathbf B(a) \boxplus n)), \quad \Gamma_{b, n} : b \boxtimes \mathbb L(n) \longrightarrow \mathbb L(\mathbf B(b) \boxplus n)
\end{equation}are morphisms in $\mathbf{Mod}_{\mathcal M}(b)$. We now have the following commutative diagrams in $\mathbf{Mod}_{\mathcal M}(b)$
\begin{equation}
\begin{tikzcd}
b \boxtimes (a \boxtimes \mathbb L(n)) \arrow[d, "{\Gamma_{b, \mathbf B(a) \boxplus n} \circ (1_b \boxtimes \Gamma_{a, n})}"'] \arrow[rrr, "{1_b \boxtimes (\mathbb L(\rho) \circ \Gamma_{a, n})}"] &  &  & b \boxtimes \mathbb L(n) \arrow[d, "{\Gamma_{b, n}}"] \\
\mathbb L(\mathbf B(b) \boxplus (\mathbf B(a) \boxplus n)) \arrow[rrr, "\mathbb L(1_{\mathbf B(b)} \boxplus \rho)"']                                                                                &  &  & \mathbb L(\mathbf B(b) \boxplus n)                   
\end{tikzcd}
\end{equation}
\begin{equation}
\begin{tikzcd}
b \boxtimes (a \boxtimes \mathbb L(n)) \arrow[rr, "\sim"] \arrow[d, "{\Gamma_{b, \mathbf B(a) \boxplus n} \circ (1_b \boxtimes \Gamma_{a, n})}"'] &  & (b \otimes a) \boxtimes \mathbb L(n) \arrow[rrr, "(\mu_b \circ (1_b \otimes \alpha)) \boxtimes 1_{\mathbb L(n)}"] &  &  & b \boxtimes \mathbb L(n) \arrow[d, "{\Gamma_{b, n}}"] \\
\mathbb L(\mathbf B(b) \boxplus (\mathbf B(a) \boxplus n)) \arrow[rr, "\sim"']                                                                    &  & \mathbb L((\mathbf B(b) \otimes \mathbf B(a)) \boxplus n) \arrow[rrr, "\mathbb L(\psi)"']                         &  &  & \mathbb L(\mathbf B(b) \boxplus n)                   
\end{tikzcd}
\end{equation}It follows from \eqref{eqtuniq5.16} and the universal property of the coequalizer
\begin{equation}
    b \boxtimes_a \mathbb L(n) = Coeq\left(
\begin{tikzcd}
b \boxtimes (a \boxtimes \mathbb L(n)) \arrow[rrrr, "{1_b \boxtimes (\mathbb L(\rho) \circ \Gamma_{a, n})}", shift left=3] \arrow[rr, "\sim"', shift right] &  & (b \otimes a) \boxtimes \mathbb L(n) \arrow[rr, "(\mu_b \circ (1_b \otimes \alpha)) \boxtimes 1_{\mathbb L(n)}"', shift right] &  & b \boxtimes \mathbb L(n)
\end{tikzcd}\right)
\end{equation}that there is a morphism $\theta_{\alpha, \mathbf B, \mathbb L, n} : b \boxtimes_a \mathbb L(n) \longrightarrow \mathbb L(\mathbf B(b) \boxplus_{\mathbf B(a)} n)$ in $\mathbf{Mod}_{\mathcal M}(b)$.

\smallskip
(2) The result is clear from the proof of (1). 
\end{proof}
\begin{Thm}\label{basechangingscheme}
Suppose that we are given the following   data:

\smallskip
(a) An adjunction $ \left(\mathbf B :  \mathcal C \longrightarrow \mathcal D,\text{ }\mathbf A : \mathcal D \longrightarrow \mathcal C\right)$ between closed symmetric monoidal categories
$\mathcal C$ and $\mathcal D$ such that the left adjoint $\mathbf B$ has a strong symmetric monoidal structure.

\smallskip
(b) A left $\mathcal C$-actegory $(\mathcal M, \boxtimes)$, a left $\mathcal D$-actegory $(\mathcal N, \boxplus)$ and a lax $\mathcal C$-linear functor 
\begin{equation}\label{actadjunctv}
   (\mathbb L, \Gamma) : \mathbf B_*(\mathcal N) = (\mathcal N, \boxplus^{\mathbf B}) \longrightarrow (\mathcal M, \boxtimes)
\end{equation} satisfying conditions described above.

\smallskip
    Suppose also that $\mathbf A : \mathcal D \longrightarrow \mathcal C$ preserves filtered colimits and the functor $\mathbb L : \mathbf B_{*}(\mathcal N) \longrightarrow \mathcal M$ is conservative and preserves finite limits. Further, suppose that for any $\mathcal M$-flat morphism $a \longrightarrow b$ in $Comm(\mathcal C)$ and any $n \in \mathbf{Mod}_{\mathcal N}(\mathbf B(a))$, the canonical morphism in \eqref{canonic}
    \begin{equation}
        b \boxtimes_a \mathbb L(n) \longrightarrow \mathbb L(\mathbf B(b) \boxplus_{\mathbf B(a)} n) 
    \end{equation}is an isomorphism in $\mathbf{Mod}_{\mathcal M}(b)$. Then,
    
    \smallskip
    (1) The functor $\mathbf B^{op} : Aff_{\mathcal C} \longrightarrow Aff_{\mathcal D}$ takes fpqc $\mathcal M$-covers to fpqc $\mathcal N$-covers.

    \smallskip
    (2) The functor $\mathbf B_! = \_\_ \circ \mathbf B : PSh(Aff_{\mathcal D}) \longrightarrow PSh(Aff_{\mathcal C})$ in \eqref{equt5.7} restricts to a functor $\widetilde{\mathbf B_!} : Sh(Aff_{\mathcal D})_{\mathcal N} \longrightarrow Sh(Aff_{\mathcal C})_{\mathcal M}$ which has a left adjoint $\widetilde{\mathbf A_!} : Sh(Aff_{\mathcal C})_{\mathcal M} \longrightarrow Sh(Aff_{\mathcal D})_{\mathcal N}$.

    \smallskip
    Additionally, if $\mathcal M$ and $\mathcal N$ are subcanonical, then

    \smallskip
    (3) The functor $\widetilde{\mathbf A_!}$ restricts to a functor $Sch(\mathcal C)_{\mathcal M} \longrightarrow Sch(\mathcal D)_{\mathcal N}$ sending $F \mapsto \widetilde{\mathbf A_!}(F)$.

    \smallskip
    For any affine $X = Spec(a) \in Aff_{\mathcal C}$, $\widetilde{\mathbf A_!}(X)$ is canonically isomorphic to the affine scheme $\mathbf B^{op}(Spec(a)) \in Aff_{\mathcal D}$.
\end{Thm}
\begin{proof}
(1) Since $\mathbb L : \mathbf B_*(\mathcal N) \longrightarrow \mathcal M$ is conservative and preserves finite limits, it follows from Remark $\ref{creation}$ that for each $a \in Comm(\mathcal C)$, the functor $\mathbb L = \mathbb L^a : \mathbf{Mod}_{\mathbf B_*(\mathcal N)}(a) \longrightarrow \mathbf{Mod}_{\mathcal M}(a)$ in \eqref{modadjunct} is also conservative and preserves finite limits. Let $\{\alpha_i^{op} : Spec(a_i) \longrightarrow Spec(a)\}_{i \in I}$ be an fpqc $\mathcal M$-cover in $Aff_{\mathcal C}$.
We need to show that
\begin{equation} 
\{\mathbf B^{op}(\alpha_i^{op}) : \mathbf B^{op}(Spec(a_i)) \longrightarrow \mathbf B^{op}(Spec(a))\}_{i \in I}
\end{equation}is an fpqc $\mathcal N$-cover in $Aff_{\mathcal D}$. By definition, each $\alpha_i : a \longrightarrow a_i$ is $\mathcal M$-flat and there exists a finite subset $J \subseteq I$ such that the functor $((\alpha_j)^*)_{j \in J} = (a_j \boxtimes_a \_\_)_{j \in J} : \mathbf{Mod}_{\mathcal M}(a) \longrightarrow \prod_{j \in J} \mathbf{Mod}_{\mathcal M}(a_j)$ is conservative. For each $i \in I$, we know that $\alpha_i:a\longrightarrow a_i$ is $\mathcal M$-flat and hence the following square
\begin{equation}\label{diagrmm5.16}
\begin{tikzcd}
\mathbf{Mod}_{\mathbf B_*(\mathcal N)}(a) = \mathbf{Mod}_{\mathcal N}(\mathbf B(a)) \arrow[rr, "\mathbb L^a"] \arrow[d, "\mathbf B(\alpha_i)^*"'] &  & \mathbf{Mod}_{\mathcal M}(a) \arrow[d, "\alpha_i^*"] \\
\mathbf{Mod}_{\mathbf B_*(\mathcal N)}(a_i) = \mathbf{Mod}_{\mathcal N}(\mathbf B(a_i)) \arrow[rr, "\mathbb L^{a_i}"']                                &  & \mathbf{Mod}_{\mathcal M}(a_i)                      
\end{tikzcd}
\end{equation}commutes up to a natural isomorphism. Since $\mathbb L^a$, $\mathbb L^{a_i}$ preserve finite limits, $\alpha_i$ is $\mathcal M$-flat (i.e., 
$\alpha_i^*$ preserves finite limits) and $\mathbb L^{a_i}$ is conservative, it follows from \eqref{diagrmm5.16} that $\mathbf B(\alpha_i)^*$ preserves finite limits, i.e.,  $ \mathbf B(\alpha_i)$ is $\mathcal N$-flat.

\smallskip
Further, since the square \eqref{diagrmm5.16} commutes up to a natural isomorphism for each $i \in J \subseteq I$, the following diagram
\begin{equation}\label{525t}
\begin{tikzcd}
\mathbf{Mod}_{\mathcal N}(\mathbf B(a)) \arrow[rr, "\mathbb L^a"] \arrow[d, "(\mathbf B(\alpha_j)^*)_{j \in J}"'] &  & \mathbf{Mod}_{\mathcal M}(a) \arrow[d, "(\alpha_j^*)_{j \in J}"] \\
\prod_{j \in J}\mathbf{Mod}_{\mathcal N}(\mathbf B(a_j)) \arrow[rr, "\prod_{j \in J}\mathbb L^{a_j}"']                &  & \prod_{j \in J} \mathbf{Mod}_{\mathcal M}(a_j)                  
\end{tikzcd}
\end{equation}also commutes up to a natural isomorphism. Since $\mathbb L^a$, $(\alpha_j^*)_{j\in J}$ and $\prod_{j \in J}\mathbb L^{a_j}$ are conservative, it follows from 
\eqref{525t} that 
\begin{equation} 
(\mathbf B(\alpha_j)^*)_{j \in J} : \mathbf{Mod}_{\mathcal N}(\mathbf B(a)) \longrightarrow \prod_{j \in J}\mathbf{Mod}_{\mathcal N}(\mathbf B(a_j))
\end{equation}is conservative. Hence, $\{\mathbf B^{op}(\alpha_i^{op}) : \mathbf B^{op}(Spec(a_i)) \longrightarrow \mathbf B^{op}(Spec(a))\}_{i \in I}$ is an fpqc $\mathcal N$-cover in $Aff_{\mathcal D}$.

\smallskip
(2) Using part (1) and Lemma \ref{lemcont}, it follows that $\mathbf B^{op} : Aff_{\mathcal C} \longrightarrow Aff_{\mathcal D}$ takes spectral $\mathcal M$-covers to spectral $\mathcal N$-covers. Since $\mathbf B^{op}$ preserves pullbacks, it follows that the functor $\mathbf B_! = \_\_ \circ \mathbf B : PSh(Aff_{\mathcal D}) \longrightarrow PSh(Aff_{\mathcal C})$ in \eqref{equt5.7} restricts to a functor
\begin{equation}
   \widetilde{ \mathbf B_!} : Sh(Aff_{\mathcal D})_{\mathcal N} \longrightarrow Sh(Aff_{\mathcal C})_{\mathcal M}
\end{equation}We define the functor $\widetilde{\mathbf A_!} : Sh(Aff_{\mathcal C})_{\mathcal M} \longrightarrow Sh(Aff_{\mathcal D})_{\mathcal N}$ as the composition
\begin{equation}\label{equat5.20}
    Sh(Aff_{\mathcal C})_{\mathcal M} \hookrightarrow PSh(Aff_{\mathcal C}) \xrightarrow{\text{ }\mathbf A_!\text{ }=\text{ }\_\_ \circ \mathbf A\text{ }} PSh(Aff_{\mathcal D}) \xrightarrow{\text{ }(\_\_)^{++}_{\mathcal N}\text{ }} Sh(Aff_{\mathcal D})_{\mathcal N}
\end{equation}where $(\_\_)^{++}_{\mathcal N}$ is the sheafification functor associated to the spectral $\mathcal N$-topology on $Aff_{\mathcal D}$. Since  $(\_\_)^{++}_{\mathcal N}$ has a right adjoint and $(\mathbf A_!,\mathbf B_!)$ is an adjoint pair as in \eqref{equt5.7}, we have  
\begin{equation}
  \begin{split}
   Sh(Aff_{\mathcal D})_{\mathcal N}(\widetilde{\mathbf A_!}(F), G) &= Sh(Aff_{\mathcal D})_{\mathcal N}(\mathbf A_!(F)^{++}_{\mathcal N}, G)\\
   &\cong PSh(Aff_{\mathcal D})(\mathbf A_!(F), G)\\ 
   &\cong PSh(Aff_{\mathcal C})(F, \mathbf B_!(G)) = Sh(Aff_{\mathcal C})_{\mathcal M}(F, \widetilde{\mathbf B_!}(G))
   \end{split}
\end{equation} for $F \in Sh(Aff_{\mathcal C})_{\mathcal M}$ and $G \in Sh(Aff_{\mathcal D})_{\mathcal N}$.  

\smallskip
(3) Since $\mathcal M$ is subcanonical, $Aff_{\mathcal C}$ is a full subcategory of $Sh(Aff_{\mathcal C})_{\mathcal M}$. For $Spec(a)\in Aff_{\mathcal C}$, we have natural isomorphisms
\begin{equation}\label{eniq5.29}
   \begin{split}
    \widetilde{\mathbf A_!}(Aff_{\mathcal C}(\_\_, Spec(a)) &= \left(Aff_{\mathcal C}(\mathbf A^{op}(\_\_), Spec(a))\right)^{++}_{\mathcal N}\\
    &\cong (Aff_{\mathcal D}(\_\_, \mathbf B^{op}(Spec(a))))^{++}_{\mathcal N}\quad[\text{using the adjunction in }\eqref{eqtunio5.3}]\\
    &\cong Aff_{\mathcal D}(\_\_, \mathbf B^{op}(Spec(a)))\qquad[\text{since }\mathcal N\text{ is subcanonical}]
    \end{split}
\end{equation}  Hence, $\widetilde{\mathbf A_!} : Sh(Aff_{\mathcal C})_{\mathcal M} \longrightarrow Sh(Aff_{\mathcal D})_{\mathcal N}$ restricts to $\mathbf B^{op} : Aff_{\mathcal C} \longrightarrow Aff_{\mathcal D}$ on $Aff_{\mathcal C} \subset Sh(Aff_{\mathcal C})_{\mathcal M}$.

\smallskip
Now, let $F \in Sch(\mathcal C)_{\mathcal M} \subseteq Sh(Aff_{\mathcal C})_{\mathcal M}$. We need to show that $\widetilde{\mathbf A_!}(F) \in Sh(Aff_{\mathcal D})_{\mathcal N}$ is an $\mathcal N$-scheme. It follows from Proposition $\ref{equivalencerel}$ that there exists a family $\{U_i\}_{i \in I}$ of affines in $Aff_{\mathcal C}$ and an equivalence relation $R \subseteq H \times H$ on $H := \coprod_{i \in I} U_i$  such that $F \cong H/R$ in $Sh(Aff_{\mathcal C})_{\mathcal M}$. Further, for each $i, j \in I$, the composition $R_{i, j} := R \times_{H \times H} (U_i \times U_j) \xrightarrow{\text{ }r_{i, j}\text{ }} U_i \times U_j \longrightarrow U_i$ is a Zariski open immersion relative to $\mathcal M$ and the subobject $R_{i, i} \subseteq U_i \times U_i$ is the image of $U_i \xrightarrow{\text{ }(1_{U_i}, 1_{U_i})\text{ }} U_i \times U_i$ in $Sh(Aff_{\mathcal C})_{\mathcal M}$.

\smallskip
Since limits of presheaves are computed objectwise, we see that $\mathbf A_! = \_\_ \circ \mathbf A  : PSh(Aff_{\mathcal C}) \longrightarrow PSh(Aff_{\mathcal D})$ preserves limits. Further, since the sheafication functor $(\_\_)^{++}_{\mathcal N}$ preserves finite limits, it follows from $\eqref{equat5.20}$ that $\widetilde{\mathbf A_!} : Sh(Aff_{\mathcal C})_{\mathcal M} \longrightarrow Sh(Aff_{\mathcal D})_{\mathcal N}$ preserves finite limits and in particular, monomorphisms. Also, since $\widetilde{\mathbf A_!}$ is a left adjoint, it preserves colimits and in particular, epimorphisms. Hence, $\widetilde{\mathbf A_!}(R) \subseteq \widetilde{\mathbf A_!}(H) \times \widetilde{\mathbf A_!}(H)$ is an equivalence relation on $\widetilde{\mathbf A_!}(H)$ and
\begin{equation}
    \widetilde{\mathbf A_!}(F) \cong \widetilde{\mathbf A_!}(H)/\widetilde{\mathbf A_!}(R) \cong \left(\coprod_{i \in I}\widetilde{\mathbf A_!}(U_i)\right)/\widetilde{\mathbf A_!}(R)
\end{equation}in $Sh(Aff_{\mathcal D})_{\mathcal N}$. We will verify that the equivalence relation $\widetilde{\mathbf A_!}(R)$ satisfies the conditions of Proposition $\ref{equivalencerel}$. Using \eqref{eniq5.29}, it follows that for each $i \in I$, $\widetilde{\mathbf A_!}(U_i) = \mathbf B^{op}(U_i) \in Aff_{\mathcal D}$ is an affine. Further,  $\widetilde{\mathbf A_!}(R_{i, i}) \subseteq \widetilde{\mathbf A_!}(U_i) \times \widetilde{\mathbf A_!}(U_i)$ is the image of $(1_{\widetilde{\mathbf A_!}(U_i)}, 1_{\widetilde{\mathbf A_!}(U_i)}) : \widetilde{\mathbf A_!}(U_i) \longrightarrow \widetilde{\mathbf A_!}(U_i) \times \widetilde{\mathbf A_!}(U_i)$ in $Sh(Aff_{\mathcal D})_{\mathcal N}$.

Now, for each $i, j \in I$, since $R_{i, j} \xrightarrow{\text{ }r_{i, j}\text{ }} U_i \times U_j \longrightarrow U_i$ is a Zariski open immersion relative to $\mathcal M$, it follows from Proposition $\ref{zaropenim}$ that its image is Zariski open relative to $\mathcal M$. Hence, there exists a family $\{\beta_{ijk}^{op} : Y_{ijk} \longrightarrow U_i\}_{k \in K_{ij}}$ of spectral immersions relative to $\mathcal M$ such that the image of the induced morphism $\coprod_{k \in K_{ij}} \beta_{ijk}^{op} : \coprod_{k \in K_{ij}} Y_{ijk} \longrightarrow U_i$ is $R_{ij}$. Since $\widetilde{\mathbf A_!}$ preserves colimits and finite limits and restricts to the functor $\mathbf B^{op} : Aff_{\mathcal C} \longrightarrow Aff_{\mathcal D}$ on the subcategory $Aff_{\mathcal C} \subset Sh(Aff_{\mathcal C})_{\mathcal M}$, it follows from part (1) and Lemma \ref{lemcont} that
\begin{equation} 
\left\{\widetilde{\mathbf A_!}(Y_{ijk}) \cong \mathbf B^{op}(Y_{ijk}) \xrightarrow{\text{ }\widetilde{\mathbf A_!}(\beta_{ijk}^{op}) = \mathbf B^{op}(\beta_{ijk}^{op})\text{ }} \mathbf B^{op}(U_i) = \widetilde{\mathbf A_!}(U_i)\right\}_{k \in K_{ij}}
\end{equation}is a family of spectral immersions relative to $\mathcal N$ such that the image of the induced morphism $\coprod_{k \in K_{ij}} \widetilde{\mathbf A_!}(Y_{ijk}) \longrightarrow \widetilde{\mathbf A_!}(U_i)$ is $\widetilde{\mathbf A_!}(R_{ij})$. Hence, the image of the monomorphism
\begin{equation}\label{equq5.31} 
\widetilde{\mathbf A_!}\left(R_{i, j} \xrightarrow{\text{ }r_{i, j}\text{ }} U_i \times U_j \longrightarrow U_i\right) = \left(\widetilde{\mathbf A_!}(R_{ij}) \xrightarrow{\text{ }\widetilde{\mathbf A_!}(r_{i, j})\text{ }} \widetilde{\mathbf A_!}(U_i \times U_j) \cong \widetilde{\mathbf A_!}(U_i) \times \widetilde{\mathbf A_!}(U_j) \longrightarrow \widetilde{\mathbf A_!}(U_i)\right)
\end{equation}is Zariski open relative to $\mathcal N$. Since $\mathcal N$ is subcanonical, it follows from Proposition $\ref{zaropenim}$ that the monomorphism in \eqref{equq5.31} is a Zariski open immersion relative to $\mathcal N$. Using Proposition $\ref{equivalencerel}$, it follows that $\widetilde{\mathbf A_!}(F)$ is an $\mathcal N$-scheme. This concludes the proof.
\end{proof}
\begin{cor}\label{coro5.4}
Let $\mathcal E$ be a closed symmetric monoidal category which is both complete and cocomplete. Let $(\mathcal P, \boxtimes^{\mathcal P})$ $(\text{resp. }(\mathcal Q, \boxtimes^{\mathcal Q}))$ be a left $\mathcal E$-actegory which is both complete and cocomplete and such that $\boxtimes^{\mathcal P} : \mathcal E \times \mathcal P \longrightarrow \mathcal P$ $(\text{resp. }\boxtimes^{\mathcal Q} : \mathcal E \times \mathcal Q \longrightarrow \mathcal Q)$ preserves colimits in both variables. Let $\mathbb S : \mathcal Q \longrightarrow \mathcal P$ be a lax $\mathcal E$-linear functor which is conservative and preserves finite limits. Further, suppose that for every $\mathcal P$-flat morphism $\alpha : a \longrightarrow b$ in $Comm(\mathcal E)$ and $n \in \mathbf{Mod}_{\mathcal Q}(a)$, the canonical morphism $\theta_{\alpha, id_{\mathcal E}, \mathbb S, n} : b \boxtimes^{\mathcal P}_a \mathbb S(n) \longrightarrow \mathbb S(b \boxtimes^{\mathcal Q}_a n)$ is an isomorphism in $\mathbf{Mod}_{\mathcal P}(b)$.

\smallskip
If $\mathcal Q$ is subcanonical, then so is $\mathcal P$ and for each $F \in Sch(\mathcal E)_{\mathcal P} \subset PSh(Aff_{\mathcal E})$, the sheafification $(F)^{++}_{\mathcal Q}$ of $F$ with respect to the spectral $\mathcal Q$-topology is a $\mathcal Q$-scheme.
\end{cor}
\begin{proof}
The result follows from an application of Theorem $\ref{basechangingscheme}$ to the adjoint pair $(id_{\mathcal E}, id_{\mathcal E})$.
\end{proof}
\section{Examples}

In this section, we give a number of examples of subcanonical spectral topologies arising from our setup. 

\begin{examples}\label{examp1}
\emph{Let $(\mathcal C, \otimes, 1)$ be a closed symmetric monoidal category which is both complete and cocomplete. Then $\mathcal C$ may be treated as a left 
$\mathcal C$-actegory.  It follows from \cite[Corollary 2.11]{TV} that representable presheaves on $Aff_{\mathcal C}$ are also  fpqc sheaves, which means that  $\mathcal C$ is subcanonical as a $\mathcal C$-actegory in the sense of Section 4. It is also clear that schemes over $\mathcal C$ treated as a left $\mathcal C$-actegory are precisely the schemes relative to $(\mathcal C, \otimes,1)$  in the sense of To\"{e}n and Vaqui\'{e} \cite{TV}.}
\end{examples}

\begin{examples}\label{examp3}
\emph{Let $(\mathcal C, \otimes, 1)$ be a closed symmetric monoidal category which is both complete and cocomplete. We fix a set $J$, considered as a discrete category. It is clear that the functor category $[J, \mathcal C] = \prod_J \mathcal C$ is both complete and cocomplete. We note that $[J, \mathcal C]$ has a canonical left $(\mathcal C, \otimes, 1)$-actegory structure $\overset{J}{\otimes} : \mathcal C \times [J, \mathcal C] \longrightarrow [J, \mathcal C]$ defined as follows : for each $c \in \mathcal C$ and $M \in [J, \mathcal C]$
\begin{equation}\label{eqew5.34}
c \overset{J}{\otimes} M : J \longrightarrow \mathcal C,\quad(J \ni j \mapsto c \otimes M(j) \in \mathcal C)
\end{equation}Since colimits in functor categories are computed objectwise and $\otimes : \mathcal C \times \mathcal C \longrightarrow \mathcal C$ preserves colimits in both variables, hence so does the bifunctor $\overset{J}{\otimes} : \mathcal C \times [J, \mathcal C] \longrightarrow [J, \mathcal C]$ given by \eqref{eqew5.34}. We will denote the left $(\mathcal C, \otimes, 1)$-actegory $([J, \mathcal C], \overset{J}{\otimes})$ simply by $[J, \mathcal C]$.}

\smallskip
\emph{It is clear that for each $a \in Comm(\mathcal C)$, $\mathbf{Mod}_{[J, \mathcal C]}(a) \cong [J, \mathbf{Mod}_{\mathcal C}(a)]$. Further, for each morphism $\alpha : a \longrightarrow b$ in $Comm(\mathcal C)$, it may be verified that the extension of scalars $b \overset{J}{\otimes}_a \_\_ : \mathbf{Mod}_{[J, \mathcal C]}(a) \longrightarrow \mathbf{Mod}_{[J, \mathcal C]}(b)$ associated to $[J, \mathcal C]$ is the composition
\begin{equation}\label{eqew5.35}
    \mathbf{Mod}_{[J, \mathcal C]}(a) \cong [J, \mathbf{Mod}_{\mathcal C}(a)] \xrightarrow{\text{ }\alpha^* \circ \_\_\text{ }} [J, \mathbf{Mod}_{\mathcal C}(b)] \cong \mathbf{Mod}_{[J, \mathcal C]}(b)
\end{equation}where $\alpha^* = b \otimes_a \_\_ : \mathbf{Mod}_{\mathcal C}(a) \longrightarrow \mathbf{Mod}_{\mathcal C}(b)$ is the extension of scalars associated to $\mathcal C$. Since limits in functor categories are computed objectwise, it follows that $\alpha : a \longrightarrow b$ in $Comm(\mathcal C)$ is $[J, \mathcal C]$-flat if and only if it is $\mathcal C$-flat. Using \eqref{eqew5.35}, it may also be verified that fpqc (resp. spectral) $[J, \mathcal C]$-covers are the same as fpqc (resp. spectral) $\mathcal C$-covers in $Aff_{\mathcal C}$. It again follows from Example \ref{examp1} that $[J, \mathcal C]$ is subcanonical. Further, $[J, \mathcal C]$-schemes are the same as $\mathcal C$-schemes.} 
\end{examples}
\begin{rem}
    \emph{Let $(\mathcal M, \boxtimes)$ be a subcanonical left $(\mathcal C, \otimes, 1)$-actegory. Using the same arguments as that of Example \ref{examp3}, it may be verified that the bifunctor $\overset{J}{\boxtimes} : \mathcal C \times [J, \mathcal M] \longrightarrow [J, \mathcal M]$ defined by 
\begin{equation}
c \overset{J}{\boxtimes} M : J \longrightarrow \mathcal M,\quad(J \ni j \mapsto c \boxtimes M(j) \in \mathcal M)\qquad \left(c \in \mathcal C\text{ and }M \in [J, \mathcal M]\right)
\end{equation}makes $[J, \mathcal M]$ into a subcanonical $(\mathcal C, \otimes, 1)$-actegory. Further, $[J, \mathcal M]$-schemes are the same as $\mathcal M$-schemes.}
\end{rem}
In what follows, we will denote the category of presheaves of sets on an arbitrary small category $\mathcal Z$ by $\widehat{\mathcal Z} := PSh(\mathcal Z)$. We note that the category $\widehat{\mathcal Z}$ is both complete and cocomplete. For $F, G \in \widehat{\mathcal Z}$, let $F \bullet G \in \widehat{\mathcal Z}$ denote their   product in $PSh(\mathcal Z)$. We recall (see for instance, \cite[$\S$ 1.6]{MM}) that $(\widehat{\mathcal Z}, \bullet, 1)$ is a cartesian closed monoidal category where $1 \in \widehat{\mathcal Z}$ denotes the terminal object. It is clear from Example \ref{examp1} that the left $(\widehat{\mathcal Z}, \bullet, 1)$-actegory $(\widehat{\mathcal Z}, \bullet)$ is subcanonical.
\begin{thm}\label{preshex}
Let $\Phi : \mathcal X \longrightarrow \mathcal Y$ be an essentially surjective functor between small categories. Then, the bifunctor
\begin{equation}\label{quaton5.37} 
\boxtimes : \widehat{\mathcal Y} \times \widehat{\mathcal X} \longrightarrow \widehat{\mathcal X},\quad(F, H) \mapsto \left((F \circ \Phi^{op}) \bullet H\right) \in \widehat{\mathcal X}
\end{equation}preserves colimits in both variables and makes $\widehat{\mathcal X}$ into a left $(\widehat{\mathcal Y}, \bullet, 1)$-actegory. Further, $(\widehat{\mathcal X}, \boxtimes)$ is subcanonical and for any $(\widehat{\mathcal X}, \boxtimes)$-scheme $F$ on $Aff_{(\widehat{\mathcal Y}, \bullet, 1)}$, the sheafification $(F)^{++}_{(\widehat{\mathcal Y}, \bullet)}$ of $F$ with respect to the spectral $(\widehat{\mathcal Y}, \bullet)$-topology on $Aff_{(\widehat{\mathcal Y}, \bullet, 1)}$ is a $(\widehat{\mathcal Y}, \bullet)$-scheme.
\end{thm}
\begin{proof}
We consider the functor 
\begin{equation}\label{67uh}
\widehat{\Phi} : \widehat{\mathcal Y} = PSh(\mathcal Y) \longrightarrow PSh(\mathcal X) = \widehat{\mathcal X},\quad \widehat{\mathcal Y} \ni F \mapsto F \circ \Phi^{op} \in \widehat{\mathcal X}
\end{equation}Since limits and colimits in presheaf categories are computed objectwise, we note that $\widehat{\Phi} = \_\_ \circ \Phi^{op}$ preserves limits and colimits. In particular, $\widehat{\Phi}$ preserves finite products so that there is a natural isomorphism
\begin{equation}\label{strong6.8}
   \delta  = \left(\delta_{F, G} : \widehat{\Phi}(F) \bullet \widehat{\Phi}(G) \xrightarrow{\text{ }\sim\text{ }} \widehat{\Phi}(F \bullet G)\text{ }\big{\vert}\text{ }F, G \in \widehat{\mathcal Y}\right)
\end{equation}It may now be verified that the bifunctor
\begin{equation}\label{quaton5.37} 
\boxtimes : \widehat{\mathcal Y} \times \widehat{\mathcal X} \longrightarrow \widehat{\mathcal X},\quad(F, H) \mapsto \left(\widehat{\Phi}(F) \bullet H\right) = \left((F \circ \Phi^{op}) \bullet H\right) \in \widehat{\mathcal X}
\end{equation}preserves colimits in both variables and makes $\widehat{\mathcal X}$ into a left $(\widehat{\mathcal Y}, \bullet, 1)$-actegory. We need to show that $(\widehat{\mathcal X}, \boxtimes)$ is subcanonical. We note that the natural isomorphism in \eqref{strong6.8}
\begin{equation}
    \left(F \boxtimes \widehat{\Phi}(G) = \widehat{\Phi}(F) \bullet \widehat{\Phi}(G) \xrightarrow{\text{ }\delta_{F, G}\text{ }}\widehat{\Phi}(F \bullet G)\text{ }\big{\vert}\text{ }F, G \in \widehat{\mathcal Y}\right)
\end{equation}makes $\widehat{\Phi} : (\widehat{\mathcal Y}, \bullet) \longrightarrow (\widehat{\mathcal X}, \boxtimes)$ into a strong $(\widehat{\mathcal Y}, \bullet, 1)$-linear functor. Since $\widehat{\Phi}$ preserves coequalizers, it follows from Lemma \ref{crucial}(2) (with $\mathcal C = \mathcal D = (\widehat{\mathcal Y}, \bullet, 1), \mathbf B = \mathbf A = id_{\widehat{\mathcal Y}}, (\mathcal M, \boxtimes) = (\widehat{\mathcal X}, \boxtimes), (\mathcal N, \boxplus) = (\widehat{\mathcal Y}, \bullet)$ and $(\mathbb L, \Gamma) = (\widehat{\Phi}, \delta)$) that for any morphism $\alpha : F \longrightarrow G$ in $Comm(\widehat{\mathcal Y})$ and any $H \in \mathbf{Mod}_{\widehat{\mathcal Y}}(F)$, the canonical morphism
\begin{equation} 
\theta_{\alpha, id_{\widehat{\mathcal Y}}, \widehat{\Phi}, H} : G \boxtimes_F \widehat{\Phi}(H) \longrightarrow \widehat{\Phi}(G \bullet_F H)
\end{equation}is an isomorphism in $\mathbf{Mod}_{\widehat{\mathcal X}}(G)$.
Since $\Phi : \mathcal X \longrightarrow \mathcal Y$ is essentially surjective, it follows that $\widehat{\Phi}$ is conservative. Using Corollary $\ref{coro5.4}$ (with $\mathcal E = \widehat{\mathcal Y}, (\mathcal P, \boxtimes^{\mathcal P}) = (\widehat{\mathcal X}, \boxtimes), (\mathcal Q, \boxtimes^{\mathcal Q}) = (\widehat{\mathcal Y}, \bullet)$ and $\mathbb S = (\widehat{\Phi}, \delta)$) and the fact that $(\widehat{\mathcal Y}, \bullet)$ is subcanonical as a left $\widehat{\mathcal Y}$-actegory, the result follows.
\end{proof}
Let $A$ be an ordinary monoid with identity element $e$, considered as a one-object category. We recall  that the category $PSh(A)$ is the category $\mathbf{Rep}(A)$ of representations (right actions) of $A$ on sets. An object of $\mathbf{Rep}(A)$ is a pair $(U, \rho)$, where $U$ is a set and $\rho : U \times A \longrightarrow U$ is the action $u.a:=\rho(u,a)$.

\begin{examples}
\emph{Let $\sigma : M \longrightarrow N$ be a morphism of monoids. Then, $\sigma$ is an essentially surjective functor between the one-object categories associated to $M$ and $N$. Using Proposition \ref{preshex}, it follows  that the bifunctor $\boxtimes : \mathbf{Rep}(N) \times \mathbf{Rep}(M) \longrightarrow \mathbf{Rep}(M)$ given by
 \begin{equation}
    (S, \rho_S) \boxtimes (T, \rho_T) = (S \times T, \rho) \in \mathbf{Rep}(M),\qquad (S, \rho_S) \in \mathbf{Rep}(N), (T, \rho_T) \in \mathbf{Rep}(M)
\end{equation}where
\begin{equation} 
\rho : (S \times T) \times M \longrightarrow (S \times T),\quad ((s, t), m)\text{ }\mapsto\text{ }(\rho_S(s, \sigma(m)), \rho_T(t, m)) = (s.\sigma(m), t.m) \in S \times T
\end{equation}preserves colimits in both variables and makes $\mathbf{Rep}(M)$ into a subcanonical left $(\mathbf{Rep}(N), \bullet, 1)$-actegory. Further, for any $F \in Sch(\mathbf{Rep}(N))_{\mathbf{Rep}(M)}$, the sheafification $(F)^{++}_{\mathbf{Rep}(N)}$ of $F$ with respect to the spectral $(\mathbf{Rep}(N), \bullet)$-topology on $Aff_{\mathbf{Rep}(N)}$ is a $\mathbf{Rep}(N)$-scheme.}
\end{examples}
\begin{examples}\label{examp8}
\emph{We recall (see for instance, \cite[$\S$ 2.7]{Mac}) that a directed graph $G$ is a tuple $(V, E, s, t)$ where $V$ and $E$ are sets and $s : E \longrightarrow V$ (resp. $t : E \longrightarrow V$) is a function, called source (resp. target). A morphism $G = (V, E, s, t) \longrightarrow G' = (V', E', s', t')$ of directed graphs is a pair
\begin{equation}
  (f : V \longrightarrow V', g : E \longrightarrow E')
\end{equation}of functions such that $s' \circ g = f \circ s$ and $t' \circ g = f \circ t$. The category of directed graphs and graph morphisms is denoted by $Digrph$.}

\smallskip
\emph{Let $M$ be a monoid and $m, n \in M$. We consider the bifunctor
\begin{equation}\label{eqap6.8}
  \overset{(m, n)}{\boxtimes} : \mathbf{Rep}(M) \times Digrph \longrightarrow Digrph,\qquad ((U, \rho), (V, E, s, t)) \mapsto (U \times V, U \times E, s', t') 
\end{equation}where  
\begin{equation}
  \begin{split}
 s' : U \times E \longrightarrow U \times V,&\quad (u, a) \mapsto (\rho(u, m), s(a)) = (u.m, s(a))\\
 t' : U \times E \longrightarrow U \times V,&\quad (u, a) \mapsto (\rho(u, n), t(a)) = (u.n, t(a))
  \end{split}
\end{equation}We claim that $(Digrph, \overset{(m, n)}{\boxtimes})$ is a subcanonical left $\mathbf{Rep}(M)$-actegory. To see this, we consider the category}
\begin{equation}
   \mathcal X = \left(
\begin{tikzcd}
0 \arrow[rr, "d_0", shift left=2] \arrow[rr, "d_1"', shift right=2] &  & 1
\end{tikzcd}\right)
\end{equation}
\emph{with exactly two objects $0, 1$ and two non-identity arrows $d_0$ and $d_1$. It may be verified that the functor
\begin{equation}
    PSh(\mathcal X) = \widehat{\mathcal X} \longrightarrow Digrph,\qquad \widehat{\mathcal X} \ni G\text{ }\mapsto\text{ }\left(G(0), G(1), G(d_0^{op}), G(d_1^{op})\right) \in Digrph
\end{equation}is an isomorphism of categories. Let $\mathcal Y$ be the one-object category associated to the monoid $M$. Let $\Phi_{m, n} : \mathcal X \longrightarrow \mathcal Y$ be the functor which sends $0$ and $1$ to the unique object of $\mathcal Y$ and the arrow $d_0$ (resp. $d_1$) to $m$ (resp. $n$). It may be verified that the functor $\widehat{\Phi_{m, n}} = \_\_ \circ \Phi_{m, n}^{op} : PSh(\mathcal Y) \longrightarrow PSh(\mathcal X)$ is given by
\begin{equation}
   \_\_ \circ \Phi_{m, n}^{op} : PSh(\mathcal Y) = \mathbf{Rep}(M) \longrightarrow Digrph \cong PSh(\mathcal X),\quad (U, \rho) \mapsto (U, U, s_U, t_U)
\end{equation}where
\begin{equation}
  \begin{split}
    s_U : U \longrightarrow U,&\quad u \mapsto \rho(u, m) = u.m\\
    t_U : U \longrightarrow U,&\quad u \mapsto \rho(u, n) = u.n
  \end{split}
\end{equation}}

\smallskip
\emph{Using \eqref{eqap6.8}, it follows that for $(U, \rho) \in \mathbf{Rep}(M)$ and $G \in Digrph$
\begin{equation} 
\widehat{\Phi_{m, n}}(U, \rho) \bullet G = (U, \rho) \overset{(m, n)}{\boxtimes} G
\end{equation}where $\bullet : Digrph \times Digrph \longrightarrow Digrph$ is the cartesian monoidal structure on $Digrph \cong PSh(\mathcal X)$. Since $\Phi_{m, n} : \mathcal X \longrightarrow \mathcal Y$ is essentially surjective, it follows from Proposition \ref{preshex} that $(Digrph, \overset{(m, n)}{\boxtimes})$ is a subcanonical left $\mathbf{Rep}(M)$-actegory.}
\end{examples}
\begin{examples}
\emph{In particular, let $M$ be the monoid with exactly one element $e$. It is clear that the category $\mathbf{Rep}(M) = PSh(M)$ is the cartesian monoidal category $Set$. It follows from Example \ref{examp8} that the bifunctor
\begin{equation}
    \overset{(e, e)}{\boxtimes} : Set \times Digrph \longrightarrow Digrph,\quad (S, G)\text{ }\mapsto\text{ }(S, S, 1_S, 1_S) \bullet G \in Digrph
\end{equation}makes $Digrph$ into a subcanonical left $Set$-actegory.}
\end{examples}
\begin{examples}
    \emph{We consider the canonical functor
\begin{equation}
    disc : Set \longrightarrow Top
\end{equation}that sends every set to the associated discrete topological space. We note that $disc$ is conservative, preserves finite limits and is left adjoint to the forgetful functor $U : Top \longrightarrow Set$. In particular, since $disc$ preserves finite products, there is a natural isomorphism 
\begin{equation}\label{cacac6.24}
    \delta = \left(disc(S) \times disc(T) \xrightarrow{\text{ }\sim\text{ }} disc(S \times T)\text{ }\big{\vert}\text{ }S, T \in Set\right)
\end{equation}such that $(disc, \delta, 1 \xrightarrow{\text{ }id\text{ }} disc(1)) : (Set, \times, 1) \longrightarrow (Top, \times, 1)$ is a strong monoidal functor. It follows that the bifunctor $\boxtimes : Set \times Top \longrightarrow Top$ given by}
\begin{equation}
\begin{split}
    Set \times Top \xrightarrow{\text{ }disc \times 1_{Top}\text{ }} &\text{ }Top \times Top \xrightarrow{\qquad\times\qquad} Top\\
    (S, X)\qquad\longmapsto \qquad &(disc(S), X)\quad\longmapsto\quad disc(S) \times X \cong \coprod_{S} X
\end{split}
\end{equation}\emph{determines a left $(Set, \times, 1)$-actegory structure on $Top$. Let $S \in Set$. Since coproducts commute with colimits, the functor
\begin{equation} 
S \boxtimes \_\_ : Top \longrightarrow Top, X \mapsto \coprod_{S} X
\end{equation}preserves colimits. Also for $X \in Top$, since the functor $\_\_ \boxtimes X : Set \longrightarrow Top$ is left adjoint to the functor $Top(X, \_\_) : Top \longrightarrow Set$, $\_\_ \boxtimes X$ preserves colimits.}

\smallskip
\emph{We claim that the left $(Set, \times, 1)$-actegory $(Top, \boxtimes)$ is subcanonical. It may be verified that $\delta$ in \eqref{cacac6.24} determines a strong $(Set, \times, 1)$-linear structure on the functor
\begin{equation} 
disc : (Set, \times) \longrightarrow (Top, \boxtimes)
\end{equation}Since $disc$ preserves coequalizers, it follows from Lemma \ref{crucial}(2) (with $\mathcal C = \mathcal D = (Set, \times, 1), \mathbf B = \mathbf A = 1_{Set}, (\mathcal M, \boxtimes) = (Top, \boxtimes), (\mathcal N, \boxplus) = (Set, \times), (\mathbb L, \Gamma) = (disc, \delta)$) that for any morphism $\alpha : A \longrightarrow B$ in $Comm(Set)$ and any $N \in \mathbf{Mod}_{Set}(A)$, the canonical morphism} 
\begin{equation}
    \theta_{\alpha, 1_{Set}, disc, N} : B \boxtimes_{A} disc(N) \longrightarrow disc(B \times_{A} N)
\end{equation}\emph{is an isomorphism in $\textbf{Mod}_{Top}(B)$. It now follows from \cite[Corollary 2.11]{TV} and from Corollary \ref{coro5.4} (with $\mathcal E = (Set, \times, 1), (\mathcal P, \boxtimes^{\mathcal P}) = (Top, \boxtimes), (\mathcal Q, \boxtimes^{\mathcal Q}) = (Set, \times), \mathbb S = (disc, \delta)$) that the left $Set$-actegory $(Top, \boxtimes)$ is subcanonical.}
\end{examples}

\end{document}